\documentclass[11pt,a4paper]{article}
\usepackage[utf8]{inputenc}
\usepackage[T1]{fontenc}
\usepackage{amsmath, amssymb, amsthm, amsfonts,mathrsfs,stackrel, xurl}
\usepackage{float}
\usepackage{geometry}
\usepackage[hidelinks]{hyperref}
\usepackage{tikz-cd}
\usepackage{enumitem}
\usepackage[normalem]{ulem}

\newtheorem{thm}{Theorem}
\newtheorem{theorem}{Theorem}[section]
\newtheorem{proposition}[theorem]{Proposition}
\newtheorem*{prop}{Proposition}
\newtheorem{lemma}[theorem]{Lemma}
\newtheorem{corollary}[theorem]{Corollary}
\newtheorem*{heuristic}{Heuristic Idea}
\newtheorem{remark}[theorem]{Remark}
\newtheorem{definition}[theorem]{Definition}

\newcommand{\B}{\mathbf{B}}
\newcommand{\C}{\mathbf{C}}

\newcommand{\N}{\mathbf{N}}
\newcommand{\Q}{\mathbf{Q}}
\newcommand{\R}{\mathbf{R}}
\newcommand{\Z}{\mathbf{Z}}

\newcommand{\cO}{\mathcal O}

\newcommand{\Fone}{\mathbf{F}_{1}}

\def\spz{\Spec \Z}
\def\spzb{\overline{\Spec \Z}}
\def\spzf{(\Spec \Z)_{\mathbf{F}_{1}}}
\newcommand{\nat}{\widehat{\N^\times}}
\newcommand{\nato}{\widehat{\N^\times_0}}
\newcommand{\nto}{{\N_{\geq 0}}}

\newcommand{\Pt}{{\rm Pts}}

\newcommand{\Hom}{{\rm{Hom}}}

\newcommand{\Aut}{\mathrm{Aut}}
\newcommand{\Spec}{\operatorname{Spec}}

\newcommand{\Pic}{\operatorname{Pic}} 

\newcommand{\picone}{{\Pic}_{*}}

\newcommand{\Grp}{\rm{\bf{Grp}}}

\def\ie{{\it i.e.\/}\ }

\def\cf{{\it cf.\/}\ }

\title{On the Absolute Geometry of $\spz$\\[0.3em]
\Large and the Fargues-Fontaine curve}
\author{Alain Connes and Caterina Consani}
\date{}

\begin{document}
\maketitle

\begin{abstract}
 \noindent A guiding principle in P.~Scholze’s $p$-adic geometry asserts that the points of $\spz$ over an algebraically closed perfectoid field of characteristic $p$ are classified, up to equivalence, by its untilts. In this paper, we give a concrete geometric realization and a generalization of this paradigm.

 \noindent We construct the absolute $\Fone$-arithmetic curve $\spzf$ by pulling back the $\Fone$-structure sheaf of the arithmetic site to $\operatorname{Spec}(\Z)$. We demonstrate that the absolute curve $\spzf$ provides a common geometric origin for fundamental structures in $p$-adic Hodge theory, complex analytic geometry, and the adelic scaling site.

\noindent The moduli space of points  of  $\spzf$ over an arbitrary perfectoid field $C$, modulo intrinsic symmetries,  canonically parameterizes the space of all perfectoid fields with the same tilt, providing a universal, characteristic-independent geometric realization of Scholze's heuristic.

 \noindent Evaluating the points of $\spzf$ over the field $\C$ of complex numbers  reveals, at each prime $p$, that the non-trivial points canonically form two principal homogeneous spaces (torsors) over the  Weil groups $W_p=\Q_p^\times$ and $W_\infty=\C^\times$. Quotienting the archimedean orbit  by the discrete Frobenius symmetries yields the complex Tate curve with modulus $q=p^{-1}$. We show that this elliptic curve canonically decomposes as the product of its real locus---which exactly recovers the adelic periodic orbit $C_p=\R_+^\times/p^\Z$---and a $p$-independent phase space that emerges naturally as a real analogue of the Fargues--Fontaine curve. In contrast, equipping the same stalk with the $p$-adic topology yields the field $\Q_p$ together with the $p$-adic Tate curve $\Q_p^\times/p^\Z$.

\end{abstract}


\paragraph{MSC 2020.}  
\href{http://www.ams.org/mathscinet/msc/msc2020.html?t=14G40&btn=Current}{14G40},      
\href{http://www.ams.org/mathscinet/msc/msc2020.html?t=14G45&btn=Current}{14G45},    
\href{http://www.ams.org/mathscinet/msc/msc2020.html?t=06F05&btn=Current}{06F05}. 
\href{http://www.ams.org/mathscinet/msc/msc2020.html?t=11R42&btn=Current}{11R42}, 
\href{http://www.ams.org/mathscinet/msc/msc2020.html?t=11M55&btn=Current}{11M55}, 


\section{Introduction}

J.~Lurie in  \cite{Lurie} reports the following heuristic idea\footnote{It is stated  as Heuristic Idea $9$ (Scholze).} attributed to P.~Scholze:
\begin{heuristic}[P. Scholze]
 Let $F$ be an algebraically closed completely valued field of characteristic $p$. Then
$$
\{\text{Untilts of } F\} / \simeq
$$
is a good replacement for the set of $F$-valued points of $\spz$. 
\end{heuristic}
In this paper we give a concrete realization of this idea.
In the recent article \cite{CC6}, we introduced the monoid $\Pic(\spz)$ of arithmetic divisor classes on the arithmetic curve $\spz$ and identified $\Pic(\spz)$   with the monoid of the points of the arithmetic site $\nat$. We upgrade this object  to a Grothendieck topos equipped with a structure sheaf defined over $\Fone$\footnote{$\Fone$ is the spherical algebra of the multiplicative pointed monoid $\{0,1\}$}.
More precisely, we define the \emph{$\Fone$-arithmetic site} as the pair:
\[
\picone := (\nato, \cO_{\Fone}),\qquad \cO_{\Fone}:=\Fone[T]
\]
where $\nato$ is the presheaf topos of sets endowed with an action of the multiplicative monoid $\nto$ of non-negative integers \cite{CC1}. The structure sheaf of $\picone$ is chosen to be the spherical algebra $\Fone[T]$  of the free commutative monoid generated by a variable $T$ (see \cite{CC2} and \cite{Xu}). This sheaf naturally belongs to the topos $\nato$ via the action of $\nto$ implemented by the universal Frobenius endomorphisms
\[
\operatorname{Frob}_k(T^n) = T^{nk}, \qquad k \in \N_{\geq 0}.
\]
Notice (see \cite[\S 2.3]{CC1}) that the category of points of the topos $\nato$ is canonically obtained from the category of points of $\nat$ by adjoining an object that is simultaneously initial and terminal. On the algebraic side, this corresponds to adjoining the trivial group $\{0\}$ to
 the category of totally ordered rank-one abelian groups—equivalently, subgroups of $\Q$—with injective morphisms. Equivalently again, one may simply view this as the category of totally ordered abelian groups isomorphic to subgroups of $\Q$, where the trivial group is naturally included.  Moreover the structure of $\nato$ is enhanced to the  geometric space $\picone$ over $\Fone$.\newline
A central definition in \cite{CC6} is the Abel–Jacobi map:
\begin{equation}\label{spzmap}
\theta: \spz \longrightarrow \Pic(\spz),
\end{equation}
which assigns to a prime $p$ the additive abelian group $\Z[1/p]$,  and to the generic point $\eta$ of $\spz$ the point $\Z$. In \cite[Theorem 5.3]{CC1}, we constructed a geometric morphism of toposes
\begin{equation}\label{spzmap1}
\Theta: \spz \longrightarrow \nato.
\end{equation}
The comparison between these two maps is given by the commutative  diagram 
\begin{equation}\label{top2ad}
\begin{tikzcd}
&\spz \arrow[dl,"\Theta"']\arrow[dr,"\theta"] & 
 \\
\nato \arrow[rr, "\kappa"]& & \Pic(\spz)
\end{tikzcd}
\end{equation}
where the map $\kappa:\Pt(\nato)\to \Pic(\spz)$ extends the canonical bijection $\Pt(\nat)\stackrel{\sim}{\to} \Pic(\spz)$ by sending the base point $\{0\}$ to the generic point $\Z$ of $\Pic(\spz)$. \vspace{.03in}

We define the \emph{$\Fone$-arithmetic curve} $\spzf$ as the pair $(\Spec(\Z), \mathcal{F})$, where the sheaf
\[
\mathcal{F} := \Theta^{-1}(\mathcal{O}_{\Fone})
\]
on $\spz$ is the pullback along the morphism $\Theta$ of the structure sheaf of $\picone$.
 Proposition~\ref{prop12} states that for each prime $p \in \spz$, the stalk of $\mathcal O_{\mathbf F_1}$ at $\Theta(p)$ pulls back to the spherical algebra:
\[
\mathcal{F}_p = \Fone[T^{\Z[1/p]_+}].
\]
This algebra  exhibits a form of ``perfection’’ at $p$, in the sense that the universal Frobenius endomorphism $x \mapsto x^p$ acts on $\mathcal{F}_p$ as an automorphism.
This local structure provides a  geometric realization of  P.~Scholze's heuristic idea.
By considering the $F$-points of  $\spzf$, we show that the \emph{locality requirement} on morphisms (see  \S \ref{locality}) from $\mathcal F$  acts as a geometric sieve, namely the ultrametric dynamics of the Frobenius action trivialize the geometry at all primes different from $p$, forcing the non-trivial component of the moduli space of local points to concentrate solely above  $p$.
Above this prime, the local morphisms from  $\mathcal{F}_p$ to $F$ canonically generate the cyclotomic embeddings implemented in $p$-adic Hodge theory. After quotienting this space by the group $\operatorname{Aut}(\Q_p) \cong \Q_p^\times$, \ie the intrinsic topological symmetries of the completed stalk, one obtains the following result: (Theorem~\ref{scholze_heuristic})

\begin{thm}[Realization of Scholze’s Heuristic]\label{scholze_heuristic0}
Let $F$ be an algebraically closed perfectoid field of characteristic $p$. The moduli space of local $F$-points of $\spzf$, modulo the canonical symmetries of the stalks, decomposes over the closed points of $\Spec(\Z)$ as follows:
\begin{enumerate}
\item At any prime different from $p$, the moduli space collapses to a single orbit.
\item   At the prime $p$, the moduli space is canonically in bijection with the space of all untilts of $F$ modulo the Frobenius.
\end{enumerate}
\end{thm}
While Theorem \ref{scholze_heuristic0} provides a direct geometric realization of Scholze’s original heuristic, the absolute nature of the $\Fone$-curve $\spzf$ leads to a substantially more universal phenomenon. Since the geometry of this curve is defined ``below $\Z$,’’ it is entirely insensitive to the characteristic of the test field. We show, in particular, that the $\Fone$-curve may be evaluated over an \emph{arbitrary} perfectoid field $C$ of residue characteristic $p$, without requiring $C$ to have characteristic $p$ or to be algebraically closed.

In this general setting, the algebraic structure of the $\Fone$-stalk at $p$ intrinsically enforces the extraction of $p$-power roots, while the locality condition rigidly constrains the image to the subgroup of principal units. This leads to the following generalization of Theorem~1 (Theorem~\ref{main}):

\begin{thm}
Let $C$ be a perfectoid field of residue characteristic $p$.
\begin{enumerate}
\item The space of $C$-points of the $\Fone$-curve $\spzf$ at the prime $p$ is canonically identified with the open unit disk
$1+\mathfrak m_{C^\flat}$
inside the tilt $C^\flat$.

\item Assume that $C$ is algebraically closed. Then the space of $C$-points of the $\Fone$-curve at $p$, modulo the intrinsic symmetries of the stalks, is canonically identified with the space of untilts of $C^\flat$ modulo Frobenius \footnote{\ie the closed points of the Fargues--Fontaine curve $\mathcal X_{C^\flat,\Q_p}$ together with the distinguished point corresponding to $C^\flat$ itself.}. 
\item For every prime $\ell\neq p$, the moduli space of $C$-points reduces to a single orbit under the action of the symmetry group $\Q_\ell^\times$.

\end{enumerate}
\end{thm}

This theorem shows that the absolute $\Fone$-geometry furnishes a universal geometric receptacle for perfectoid fields. It accommodates fields of arbitrary characteristic, canonically recovers their tilt through the local arithmetic of the stalk, and produces the closed points of the Fargues–Fontaine curve precisely at the corresponding prime.\vspace{.03in}

In Section~\ref{sect4} we  further illustrate the universality of $\spzf$, by evaluating its  points in the field $\C$ of complex numbers. At a fixed prime $p$, the absolute geometry naturally detects two distinct places,
$\Sigma_p=\{p,\infty\}$,
corresponding respectively to the $p$-adic and archimedean topologies on the stalk of the structure sheaf. 

For $w\in\N^\times$, the power map $z\mapsto z^w$ on the target space $\C$ induces a ``Frobenius-type’’ action of $\N^\times$ on points. 
While this action is canonically defined for integers $w\in\N^\times$, its extension to continuous symmetries immediately encounters the familiar multivalued phenomena inherent in complex exponentiation.
Remarkably, the strict locality condition at a place $v \in \Sigma_p$ canonically resolves this ambiguity. The topology forces the morphisms to factor through the appropriate completions, extending the  action of $\N^\times$ to the  local Weil groups $W_\infty = \C^\times$ and $W_p = \Q_p^\times$ (Proposition~\ref{prop:archimedean_points_C} and Proposition~\ref{prop:nonarchimedean_points_C}): 

\begin{thm}\label{WeilGroupAction}
Let  $\spzf$ be evaluated over $\C$ at the prime $p$. For each place
$v\in\Sigma_p=\{p,\infty\}$,
the locality condition canonically extends the Frobenius action of $\N^\times$ to a continuous and well-defined action of the local Weil group $W_v$. 

Moreover, the space $\mathcal M_p^v$ of non-trivial local points is a principal homogeneous space under $W_v$; equivalently, it consists of a single $W_v$-orbit.
\end{thm}

This unified torsor structure immediately gives rise to the fundamental geometric quotients. At the archimedean place, quotienting the $W_\infty$-torsor $\mathcal M_p^\infty$ by the discrete arithmetic symmetries generated by Frobenius,
$p^\Z \subset W_\infty$,
produces the complex Tate curve
\begin{equation}\label{tatec}
E_p:=\mathcal M_p^\infty/p^\Z \cong \C^\times/p^\Z
\end{equation} 
which canonically carries the structure of an elliptic curve over $\C$.
Furthermore, the canonical action of the complex conjugation on the archimedean points induces a natural real structure on $E_p$. The periodic orbit $C_p$ of the adelic geometry (\cite{CC3,CC6}) then appears canonically as the quotient
$E_p(\R)/\langle \pm 1\rangle$,
where the  element
$-1\in W_\infty$
geometrically exchanges the two connected components of the real locus.

Alternatively, one may quotient by the continuous symmetries of the completed stalk. At the level of the Weil group, this corresponds to the subgroup fixed by complex conjugation
$\sigma\in \operatorname{Aut}(W_\infty)$,
namely
$W_\infty^\sigma=\R^\times$.
Quotienting the torsor by this subgroup collapses the geometry to the real projective line
\begin{equation}\label{ffr}
\mathcal{X}_\infty:=\mathcal M_p^\infty/W_\infty^\sigma \cong \mathbb{P}^1(\R).
\end{equation} 
This projective space thus arises naturally as an archimedean analogue of the Fargues–Fontaine curve. We let $\tilde{\mathcal{X}}_\infty$ be the canonical unramified double cover of  \eqref{ffr},   with the deck transformation given by  $ -1$.\vspace{.03in}

\begin{thm}[Geometric Decomposition of the  Tate Curve]\label{thm:tate_decomposition}
The complex Tate curve $E_p$ (as in \eqref{tatec}) 
 admits a canonical global decomposition:
\[
E_p \cong C_p \times \tilde{\mathcal{X}}_\infty.
\]
 The canonical nowhere-vanishing holomorphic $1$-form on $E_p$ is given by:
\[
\omega = \frac{d\lambda}{\lambda} + i d\theta
\]
which uniquely intertwines the  $1$-form $d\lambda/\lambda$ of the periodic orbit $C_p$ of length $\log p$ with the $1$-form $d\theta$  of the $p$-independent geometric factor $\tilde{\mathcal{X}}_\infty$ of length $2\pi$.
\end{thm}
In \cite{CC3} we developed the characteristic $1$ geometry of the periodic orbit $C_p$  as an analogue of an elliptic curve. 
In Section \ref{outlook} we point out that the  characteristic $1$ geometry of the scaling site $\mathscr S$ \cite{CC3} is intrinsically encoded in the analytic structure underlying the Fargues–Fontaine curve, and appears as the intrinsic idempotent skeleton underlying $p$-adic Hodge geometry. The uncompleted period ring carries a natural family of Gauss valuations $v_s$, parametrized by $s\in(0,\infty)$ and interpreted as a logarithmic radius. A fundamental property of these valuations is that, for every analytic function $f$, the profile
$s\mapsto v_s(f)$
is concave, piecewise linear, and possesses integral slopes. Passing to the negative valuation canonically produces a convex, piecewise affine function with integral slopes, which is precisely the datum defining a global section of the structure sheaf of the scaling site.


\section{Lift of the Abel-Jacobi Map to $\picone$}\label{abmap}

As established in our recent work \cite{CC6}, the points of the geometric space $\Pic(\spz)$ admit a natural interpretation in terms of arithmetic divisor classes on the arithmetic curve $\Spec\Z$.

An arithmetic divisor determines a non trivial subgroup of $\Q$, while its divisor class modulo principal divisors corresponds to an isomorphism class of torsion-free abelian groups of rank one. Since the points of the topos $\nato$ are classified by isomorphism classes of totally ordered abelian groups isomorphic to subgroups of $\Q$, there is a canonical surjective map
\[
\kappa:\Pt(\nato)\longrightarrow \Pic(\spz),
\]
appearing in diagram~\eqref{top2ad}. This map forgets the ordering, sending a totally ordered group to its underlying group, while the distinguished base point $0$ of the topos is mapped to the $0$ divisor \ie the group  $\Z$.\vspace{.03in}

This interpretation makes it possible to identify the Abel–Jacobi map $\theta$ introduced in \cite{CC6} with the geometric morphism $\Theta$ constructed in \cite{CC1}. More precisely, one obtains the commutative triangle of diagram~\eqref{top2ad}, which expresses the identity
$\kappa\circ\Theta=\theta$
at the level of points.

\subsection{The lift of the Abel-Jacobi map}\label{sect9.1}

We describe the map from the points of the Zariski topos $\spz$ to the points of $\nato$ induced, via Theorem 5.3 of \cite{CC1}, by the geometric morphism
\[
\Theta:\spz\longrightarrow\nato.
\]
We then show that this map lifts the \emph{Abel–Jacobi map}
\[
\theta:\Spec\Z\longrightarrow\Pic(\spz),
\]
thereby providing a topos-theoretic realization of the Abel–Jacobi map.

\begin{proposition}
$(i)$~The geometric morphism  
\[
\Theta: \Spec \Z \longrightarrow \nato
\]
assigns to each  point of $\Spec \Z$ the  isomorphism class of a specific ordered abelian group as follows:
\begin{enumerate}
    \item To the generic point $\eta\in \Spec \Z$, it associates the isomorphism class of the trivial group: $\Theta(\eta) = [\{0\}]$.
    \item To a closed point corresponding to a rational prime $p$, it associates the isomorphism class of the ordered additive group of the ring $\Z[1/p]$ (the rational numbers whose denominators are powers of $p$): $\Theta(p) = [(\Z[1/p], \Z[1/p]_+)]$.
\end{enumerate}
$(ii)$~The maps of points fulfill the equality 
\[
\kappa\circ\Theta=\theta, \qquad   \Pt(\nato)\stackrel{\kappa}{\longrightarrow} \Pic(\spz)
\]
where $\kappa$ associates to a rank one ordered abelian group the divisor class of the underlying group, and to the trivial group $\{0\}$ the zero divisor class.
\end{proposition}
\begin{proof} $(i)$~Follows from \cite[Theorem 5.3]{CC1}.\newline
$(ii)$~Follows from the definition of the Abel-Jacobi map given in \cite{CC6}. \end{proof}

\subsection{Pullback of the absolute structure sheaf}

We now consider the geometric pullback of the structure sheaf $\mathcal{O}_{\Fone}:=\Fone[T]$ of the $\Fone$-arithmetic site 
\[
\picone := (\nato, \Fone[T])
\]
along the geometric morphism $\Theta$. 
The stalk of $\mathcal{O}_{\Fone}$ at a point of  $\nato$ classified by an ordered group $(H, H_+)$ is $\Fone[T^{H_+}]$. \vspace{.03in}

Pulling back the absolute sheaf $\mathcal{O}_{\Fone}$ to classical arithmetic spaces yields a remarkable sheaf of monoids that exhibits ``partial perfection'' with respect to the universal Frobenius endomorphisms.

\begin{proposition}\label{prop12}
Let \[
\mathcal{F} := \Theta^{-1}(\mathcal{O}_{\Fone})
\]
be the pullback of the absolute structure sheaf to $\Spec \Z$. The stalks of $\mathcal{F}$ are described as follows:
\begin{enumerate}
    \item At the generic point $\eta\in \spz$: $\mathcal{F}_\eta = \Fone$.
    \item At a closed point of $\spz$ corresponding to a  prime $p$: $\mathcal{F}_p = \Fone[T^{\Z[1/p]_+}]$.
\end{enumerate}
\end{proposition}

\begin{proof}
We refer to \cite{Xu} for details of the proof. The extension from the topos $\nat$ to $\nato$ is straightforward.
\end{proof}

Algebraically, the stalk $\mathcal{F}_p$ is the  localization of the spherical monoid  algebra $\Fone[T]:=\Fone[T^{\Z_+}]$ with respect to the $p$-th power Frobenius map $\Psi_p(x) = x^p$. In $\mathcal{F}_p$, the equation $x^{p^k} = y$  has always a unique solution for any $k \ge 0$, meaning that the monoid of the spherical algebra $\Fone[T^{\Z[1/p]_+}]$ is \emph{perfect at $p$}, while it remains imperfect at any other prime different from $p$.

\subsection{Morphisms from stalks to $\Fone$-algebras}

Before evaluating the points of the $\Fone$-curve $\spzf:=(\spz,\mathcal F)$ in specific $\Fone$-algebras, it is necessary to describe the algebraic structure of the morphisms from its local stalks. At the prime $p$, the structure sheaf of $\spzf$ is governed by the $\Fone$-algebra $\Fone[T^{H_p^+}]$, where $H_p^+ := \Z[1/p]_+$. Let $A$ be an arbitrary  $\Fone$-algebra, and let $A(1_+)$ denote its underlying multiplicative monoid. By the base-change adjunction for $\Fone$-algebras (\cf \cite[Proposition 2.2]{CC4}), an $\Fone$-algebra morphism $\rho: \Fone[T^{H_p^+}] \to A$ is uniquely determined by its restriction to the generating monoid. This establishes a canonical bijection between $\Fone$-algebra morphisms and multiplicative monoid homomorphisms:
\[
\chi : (\Z[1/p]_+, +) \longrightarrow A(1_+), \quad \text{with} \quad \chi(0)=1.
\]
Because the target monoid $A(1_+)$ is commutative, the set of such homomorphisms naturally inherits an algebraic structure of its own.

\begin{proposition}\label{prop:monoid_structure_morphisms}
Let $A$ be a commutative $\Fone$-algebra. The set of $\Fone$-algebra morphisms \sloppy $\operatorname{Hom}_{\Fone}(\Fone[T^{H_p^+}], A)$ is canonically endowed with the structure of a commutative monoid with absorbing element, where the operation is given by pointwise multiplication:
\[
(\chi_1 \cdot \chi_2)(a) := \chi_1(a) \cdot \chi_2(a) \quad \text{for all } a \in \Z[1/p]_+.
\]
Within this structure:
\begin{enumerate}
    \item The neutral element $\mathbf{1}$ is the constant map $1$, defined by $\mathbf{1}(a) = 1$ for all $a \in \Z[1/p]_+$.
    \item The absorbing element $\mathbf{0}$ is the ``constant map $0$'', defined by $\mathbf{0}(0) = 1$ and $\mathbf{0}(a) = 0$ for all $a > 0$.
\end{enumerate}
\end{proposition}

\begin{proof}
Let $\chi_1, \chi_2 : \Z[1/p]_+ \to A(1_+)$ be two monoid homomorphisms. Because $A(1_+)$ is commutative, their pointwise product $\chi = \chi_1 \cdot \chi_2$ is again a monoid homomorphism:
\[
\chi(a+b) = \chi_1(a+b)\chi_2(a+b) = \chi_1(a)\chi_1(b)\chi_2(a)\chi_2(b) = \chi(a)\chi(b),
\]
and $\chi(0) = \chi_1(0)\chi_2(0) = 1 \cdot 1 = 1$. This operation is clearly associative and commutative. 

The map $\mathbf{1}(a) = 1$ trivially satisfies $\mathbf{1}(a+b) = \mathbf{1}(a)\mathbf{1}(b)$ and acts as the neutral element under pointwise multiplication. 

For the absorbing element $\mathbf{0}$, we verify it is a valid homomorphism: if $a, b > 0$, then $a+b > 0$, and $\mathbf{0}(a+b) = 0 = 0 \cdot 0 = \mathbf{0}(a)\mathbf{0}(b)$. If $a > 0$ and $b = 0$, then $\mathbf{0}(a+0) = \mathbf{0}(a) = 0 = 0 \cdot 1 = \mathbf{0}(a)\mathbf{0}(0)$. Thus $\mathbf{0}$ is a valid morphism. For any homomorphism $\chi$, the product $(\chi \cdot \mathbf{0})(a)$ yields $1 \cdot 1 = 1$ at $a=0$, and $\chi(a) \cdot 0 = 0$ for $a > 0$, proving that $\mathbf{0}$ is indeed the absorbing element of the Hom-set.
\end{proof}


\section{$\Spec\Z$ and the Fargues-Fontaine Curve}

In this section we show that the local geometry of the $\Fone$-curve $\spzf=(\Spec(\Z), \mathcal{F})$  naturally detects the (closed points of the) Fargues--Fontaine curve \cite{FF}. Starting from a point of  $\nato$, encoded by an ordered rank-one abelian group $H$, we first determine the set of places $v$ of $\Q$ at which $H$ admits a canonical topology, and we use this result to define the notion of a \emph{morphism} from $H$ to be \emph{local} at  $v$. This locality condition is then shown to be equivalent to the existence of a unique extension through the completion $\Q_v$.

We then localize by considering the stalk of $\mathcal{O}_{\Fone}=\Fone[T]$ at (the point of $\nato$ corresponding to) $H_p=\Z[1/p]$ and evaluate its $\Fone$-points in an algebraically closed perfectoid field $F$ of characteristic $p$. We prove that all such points are parameterized by $F$, whereas the locality condition cuts out exactly the open unit disk $1+\mathfrak m_F$. Using J. Lurie's description of untilts via cyclotomic embeddings \cite{Lurie}, we identify these local $\Fone$-points with the untilts of $F$ endowed by an embedding of $\Q_p^{\mathrm{cyc}}$.

In the last part of the section, we analyze the natural action of the automorphism group $\Aut(\Q_p)\cong \Q_p^\times$ of the completed stalk. The quotient by $\Z_p^\times$ forgets the choice of cyclotomic coordinate and yields the isomorphism classes of untilts of $F$, while the further quotient by $p^\Z$ produces the Frobenius orbits, hence the closed points of the Fargues--Fontaine curve. In this way, the curve $\spzf$ arises canonically as the moduli space of local $\Fone$-points modulo the intrinsic symmetries of the arithmetic site $\nato$. 

 Finally, in \S\S \ref{ffcurve} and \ref{heuristic} we show that the selection of the pullback sheaf $\mathcal F$ on $\spz$ provides a geometric realization of P. Scholze's heuristic idea.

\subsection{Dense embeddings in local fields}\label{locality}

We begin by analyzing the local fields obtained as completions of a point of the $\Fone$-arithmetic site
$
\picone=(\nato,\Fone[T])$, 
represented by a totally ordered group $H$ isomorphic to a subgroup of $\Q$. This analysis naturally identifies the places of $\Q$ that are intrinsically encoded by $H$.

Let $H\neq \{0\}$. Up to isomorphism, 
we may assume that 
\[
\Z\subseteq H\subseteq\Q.
\]
Unless $H$ is isomorphic to $\Z$, the group $H$ is dense in 
$\R$ under the standard embedding.\vspace{.03in}

In order to formulate a general theory of \emph{local morphisms} from the stalks of $\picone$ to perfectoid fields \cite{Scholze}, we begin by characterizing the local fields that can occur as completions of dense embeddings of $H$.\vspace{.03in}

We say that a prime $p$ \emph{divides} $H$ if $H$ is 
$p$-divisible.

\begin{lemma}\label{lembed}
Assume that $H\neq{0}$ is not isomorphic to $\Z$, and let $K$ be a local field \footnote{A locally compact, non-discrete topological field.}. Then $H$ admits a dense embedding into $K$ if and only if $K$ is topologically isomorphic to $\R$ or to $\Q_p$ for a prime $p$ dividing $H$.
\end{lemma}

\begin{proof}
We first prove necessity ($\Rightarrow$). 
Assume there exists a homomorphism 
$\phi: H \to K$ with dense image.

\emph{Step 1.}
Suppose $\operatorname{char}(K) = \ell > 0$, 
\ie $\ell\cdot 1_K = 0$. 
For any $x \in H$, we have
$\ell\phi(x)=\phi(\ell x)=0$, 
so $\phi(\ell H)=0$. 
Because $\Z\subsetneq H\subseteq\Q$, 
the quotient group $H/\ell H$ is finite 
(it is isomorphic to either $0$ or $\Z/\ell\Z$). 
Consequently $\phi(H)$ is finite. 
Since $\phi(H)$ is dense in $K$, 
the field $K$ would then be finite and discrete, 
contradicting the hypothesis that $K$ is a 
local field (non-discrete). 
Thus $\operatorname{char}(K)=0$.

\emph{Step 2.}
Since $\operatorname{char}(K)=0$, 
the field $K$ is a topological vector space over $\Q$. 
If $\phi(1)=0$, then $\phi(\Z)=0$. 
Because $H/\Z$ is torsion while $K$ is torsion-free, 
this would imply $\phi(H)=0$, contradicting density. 
Thus $c:=\phi(1)\neq0$. 

Replacing $\phi$ by the scaled homomorphism 
$x\mapsto c^{-1}\phi(x)$ (which preserves density), 
we may assume without loss of generality that 
$\phi(1)=1_K$. 
Using additivity, for any integer $n\in\Z$,
\[
\phi(n)=n\cdot 1_K.
\]
For any rational $a/b\in H$, we have
\[
b\,\phi(a/b)=\phi(a)=a\cdot 1_K.
\]
Since division is uniquely defined in $K$, 
it follows that 
\[
\phi(a/b)=(a/b)\cdot 1_K.
\]
Thus $\phi$ coincides with the restriction 
of the canonical embedding $\Q\hookrightarrow K$.

\emph{Step 3.}
Because $\phi(H)$ is dense in $K$, 
the subfield $\Q$ is dense in $K$. 
Since $K$ is a locally compact completion of $\Q$, 
Ostrowski's Theorem 
(\cite[Theorem~5, Chapter~I, \S3]{Weil}) 
implies that $K$ is isomorphic to either $\R$ 
(with the archimedean absolute value) 
or to $\Q_p$ (with the $p$-adic absolute value).

\emph{Step 4.}
Suppose $K\cong\Q_p$. 
We determine for which primes $p$ 
the subgroup $H$ is dense. 
If $p$ does not divide $H$, then $H$ is not 
$p$-divisible. 
Thus there exists a maximal exponent 
$n\ge0$ such that no element of $H$ 
contains $p$ in its denominator 
to a power exceeding $n$. 
Consequently every element of $H$ 
is of the form
\[
p^{-n}\frac{a}{b},
\quad b\ \text{coprime to }p,
\]
and therefore
\[
H\subseteq p^{-n}\Z_p.
\]
Since $p^{-n}\Z_p$ is a closed proper subgroup 
of $\Q_p$, this contradicts the density of $H$. 
Thus $p$ must divide $H$.\vspace{.02in}

We now prove sufficiency ($\Leftarrow$). 
If $K\cong\R$, the canonical embedding of $H$ 
is dense because $H\neq\Z$. 
If $K\cong\Q_p$ and $p$ divides $H$, 
then $H$ contains $\Z[1/p]$. 
Since $\Z[1/p]$ is dense in $\Q_p$, 
the canonical embedding of $H$ into $\Q_p$ 
has dense image.
\end{proof}

\subsection{Places associated to points of $\nato$}

By applying Lemma~\ref{lembed}, one may attach to each subgroup $H\subseteq \Q$ a canonical set of places of $\Q$ that serve as local geometric directions along which the stalks of $\picone$ can be evaluated.

\begin{definition}\label{setplaces}
Let $H$ be the ordered group associated with a point of $\nato$. We define the set $\Sigma_H$ of \emph{places of $H$} as the collection of places $v$ of $\Q$ for which $H$ admits a dense embedding into the corresponding local field $\Q_v$, with the convention that $\Q_\infty=\R$.
\end{definition}

The following result provides a complete description 
of the places $v\in\Sigma_H$.

\begin{proposition}\label{classpla}
Let $H$ be the ordered group associated to a point of $\nato$. 
\begin{enumerate}
    \item[(i)] Any local field containing $H$ as a dense additive subgroup is canonically isomorphic to $\Q_v$ for some place $v\in\Sigma_H$.
    \item[(ii)] One has $\Sigma_H=\emptyset$ if and only if $H=\{0\}$ or 
    $H\cong\Z$. Otherwise, the archimedean place 
    $\infty$ always belongs to $\Sigma_H$.
    \item[(iii)] The non-archimedean places
    $
    \Sigma_H\setminus\{\infty\}
    $
    coincide with the set of primes $p$ dividing $H$
    (\ie such that $H$ is $p$-divisible).
    \item[(iv)] The subset
    $
    S:=\Sigma_H\setminus\{\infty\}
    $
    is the largest set of primes for which $H$ contains a subgroup isomorphic to
    $
    \Z_S=\Z[S^{-1}]$.
\end{enumerate}
\end{proposition}

\begin{proof}
$(i)$ This follows immediately from 
Lemma~\ref{lembed}.

$(ii)$ One has  $\Sigma_H=\emptyset$ if $H=\{0\}$ or 
    $H\cong\Z$. Conversely, since $H$ is isomorphic to a subgroup 
of $\Q$, it embeds into $\R$. 
Since $H$ has rank one, any two embeddings into 
$\R$ differ by multiplication by a non-zero scalar. 
If the image of such an embedding is not dense, 
it must be discrete, hence a lattice in $\R$, 
which implies $H \cong \Z$. 
Thus, if $H\neq\{0\}$, $H \not\cong \Z$, 
its image in $\R$ is dense, 
and therefore $\infty \in \Sigma_H$.

$(iii)$ This follows directly from the 
$p$-adic density criterion established 
in Lemma~\ref{lembed}.

$(iv)$ Let 
$S = \Sigma_H \setminus \{\infty\}$. 
Up to isomorphism, we may assume 
$H = H_N$ for a supernatural number 
$N = \prod p^{n_p}$ 
(where $n_p \in \Z_{\ge 0} \cup \{\infty\}$), 
defined by
\[
H_N 
= \{ x \in \Q 
\mid v_p(x) \ge -n_p 
\text{ for all primes } p \}.
\]

By $(iii)$, $H$ is $p$-divisible 
if and only if $p \in S$. 
In terms of the supernatural valuation, 
$p$-divisibility is equivalent to 
$n_p = \infty$. 
Thus $n_p=\infty$ for all $p \in S$, 
which implies that 
$\Z_S \subseteq H_N$. 

Conversely, if $H$ contains a subgroup 
isomorphic to $\Z_S$, then 
$n_p=\infty$ for all $p \in S$, 
so $H$ is $p$-divisible for each 
$p \in S$, and therefore 
$S \subseteq \Sigma_H \setminus \{\infty\}$.
\end{proof}

We may therefore describe explicitly the places of any 
non-trivial subgroup $H\subset\Q$:
\[
\Sigma_H
=
\{\infty\}
\cup
\{\,p \text{ prime}\mid H \text{ is } p\text{-divisible}\,\}.
\]
This classification naturally equips the group $H$
with a canonical family of local geometric structures.

\begin{corollary}
For each place $v\in\Sigma_H$, let
\[
i_v:H\hookrightarrow\Q_v
\]
denote the canonical embedding obtained from the inclusion
$
H\subset\Q\subset\Q_v$. 
Then any other dense embedding
\[
j_v:H\hookrightarrow\Q_v
\]
differs from $i_v$ by multiplication by a non-zero scalar. More precisely, there exists
$
\lambda\in\Q_v^\times
$
such that
\[
j_v(x)=\lambda\, i_v(x)
\qquad
\forall x\in H.
\]
In particular, the uniform structure and topology $\tau_v$ induced on $H$ by $j_v$ coincide with those induced by the canonical embedding $i_v$.
\end{corollary}

\begin{proof}
Because $H$ is a subgroup of $\Q$, 
any additive homomorphism 
$j_v: H \to \Q_v$ 
is uniquely determined by its value 
on any non-zero element. 
Let $\lambda = j_v(1)$ 
(or equivalently $\lambda = j_v(a)/a$ 
for any non-zero $a \in H$). 
By $\Q$-linearity 
(as established in the proof of 
Lemma~\ref{lembed}), 
\[
j_v(x) 
= \lambda x 
= \lambda\, i_v(x).
\]

Since multiplication by 
$\lambda \in \Q_v^\times$ 
is a bi-Lipschitz homeomorphism 
of the topological field $\Q_v$, 
the pullback of the uniform structure 
along $j_v$ coincides with that 
along $i_v$.
\end{proof}

The preceding corollary shows that for each place
$v\in\Sigma_H$, the abstract group $H$ carries a canonical topology
$\tau_v$, making it into a well-defined topological group independent of any auxiliary choices. This provides an intrinsic notion of continuity at the place $v$ of $\Q$, and therefore allows one to formulate precisely what it means for a morphism from the related stalk $\Fone[T^{H_+}]$ of $\cO_{\Fone}$ to a topological group to be \emph{continuous, or local}, at $v$.

\subsection{Local morphisms}

Having determined the intrinsic places $\Sigma_H$ of $H$, we turn to the corresponding notion of locality for morphisms from the stalk $\Fone[T^{H_+}]$ of $\cO_{\Fone}$. We prove that this locality condition is precisely characterized by extension to the associated local completion.\vspace{.03in}

Let $G$ be a complete Hausdorff topological group, 
and let $H$ be the subgroup associated to a point 
of the arithmetic site $\widehat{\N^\times_0}$. 
Consider a group homomorphism
\[
\rho:H\to G.
\]

\begin{definition}\label{defnloc}
Let $v\in\Sigma_H$ be a place of $H$. 
The homomorphism $\rho:H\to G$ 
is said to be \emph{local at $v$} 
if it is continuous with respect to the canonical topology $\tau_v$ on $H$.
\end{definition}

As we now show, this continuity condition is equivalent to a rigid extension property: namely, the homomorphism $\rho$ extends uniquely to the local completion $\Q_v$.

\begin{proposition}
A homomorphism $\rho: H \to G$ is local at $v$ 
if and only if there exists a continuous homomorphism 
$\tilde{\rho}: \Q_v \to G$ such that the diagram
\[
\begin{tikzcd}
H \arrow[r, "\rho"] \arrow[d, "i_v"', hook] & G \\
\Q_v \arrow[ru, "\tilde{\rho}"'] &
\end{tikzcd}
\]
commutes, where $i_v$ denotes the canonical dense embedding. 
Moreover, since $H$ is dense in $\Q_v$, such an extension 
$\tilde{\rho}$ is unique whenever it exists.
\end{proposition}

\begin{proof}
If the extension $\tilde{\rho}$ exists, then 
$\rho = \tilde{\rho} \circ i_v$. 
Since both $i_v$ (which induces the topology $\tau_v$) 
and $\tilde{\rho}$ are continuous, their composition 
$\rho$ is continuous with respect to $\tau_v$.

Conversely, assume that $\rho$ is continuous with respect 
to $\tau_v$. Because $\rho$ is a group homomorphism between 
topological groups, continuity at the identity implies 
uniform continuity with respect to the canonical uniform 
structures. The target group $G$ is complete and Hausdorff 
by hypothesis, and the domain $H$ (equipped with $\tau_v$) 
is a dense uniform subspace of the complete metric space 
$\Q_v$. By the universal property of completions of uniform 
spaces, the uniformly continuous map $\rho$ extends uniquely 
to a continuous map 
$\tilde{\rho}: \Q_v \to G$. 

Finally, since group operations are continuous, 
the extension $\tilde{\rho}$ is necessarily a group 
homomorphism.
\end{proof}

The notion of locality for a homomorphism 
$\rho: L \to G$ as introduced in 
Definition~\ref{defnloc} provides the correct 
geometric structure needed to evaluate the stalks 
of the $\Fone$-arithmetic site $\picone$ at complete topological fields 
$K$ (for instance, a perfectoid field of characteristic $p$).

Recall that a morphism of $\Fone$-algebras 
$\Fone[T^{L_+}]\to HK$ corresponds to a multiplicative 
monoid homomorphism 
$\phi: L_+ \to K^\times$ (\cite[Proposition 2.2]{CC4}). 
Extending $\phi$ to the Grothendieck group $L:=G(L_+)$ 
yields a group homomorphism 
$\rho:L \to K^\times$.\vspace{.03in}

Stating that the geometric point $\phi$ is 
\emph{local at $p$} means precisely that the associated 
homomorphism $\rho$ is continuous with respect 
to $\tau_p$, and therefore extends uniquely 
to the $p$-adic completion $\Q_p$.\vspace{.05in}

Later in this section we show that this locality condition provides the exact dynamical mechanism forcing the image of the stalk of the structure sheaf of $\Fone$-arithmetic site at a given point to land inside the set of closed points of the Fargues--Fontaine curve.

\subsection{Points of $\spzf$ over a field and the tilt}

Before imposing the analytic locality conditions required to evaluate the $\Fone$-curve \sloppy $\spzf=(\spz,\mathcal F)$ in a perfectoid field, it is worth  analyzing the points of that curve evaluated in a general field $F$ endowed with the coarse topology. This purely algebraic perspective reveals that the fundamental operation of $p$-adic Hodge theory---tilting---is intrinsically encoded into the absolute stalks of $\spzf$.\vspace{.03in}

Let $F$ be an arbitrary field. We consider the algebraic points of $\spzf$ at the prime $p$ with values in $F$. By the base-change adjunction (\cite[Proposition 2.2]{CC4}), a $\Fone$-algebra morphism 
$\rho: \mathcal F_p=\Fone[T^{H_p^+}] \to HF$ corresponds uniquely to a 
multiplicative monoid homomorphism 
\[
\chi : (\Z[1/p]_+, +) \longrightarrow (F, \cdot), \ \ \chi(0)=1.
\]
To understand the moduli space of such morphisms, we evaluate $\chi$ on the \emph{canonical generators of the stalk}. For each integer $n \ge 0$, let:
\[
x_n := \chi\left(\frac{1}{p^n}\right) \in F.
\]
Because $\chi$ is a monoid homomorphism, it must respect the additive structure of the stalk. The relation $\frac{1}{p^{n-1}} = p \cdot \frac{1}{p^n}$ in $\Z[1/p]_+$ translates multiplicatively in $F$ to the condition:
\[
x_{n-1} = \chi\left(\frac{1}{p^{n-1}}\right) = \chi\left(\frac{p}{p^n}\right) = \left( \chi\left(\frac{1}{p^n}\right) \right)^p = x_n^p.
\]
Consequently, the homomorphism $\chi$ is \emph{uniquely determined} by the sequence $(x_0, x_1, x_2, \dots)$ of elements in $F$ satisfying the recursive relation $x_n^p = x_{n-1}$. Such sequences are precisely the elements of the inverse limit of $F$ under the $p$-th power map. Conversely, given such a sequence one extends it multiplicatively defining:
\[
\chi(a/p^k)=x_k^a
\quad\text{for all } 
a\in\Z_{\ge0}
\] 
and one obtains in this way a  monoid homomorphism $\chi : (\Z[1/p]_+, +) \longrightarrow (F, \cdot)$.
 In the foundational language of perfectoid geometry, the inverse limit 
\[
F^\flat := \varprojlim_{ x \mapsto x^p} F
\]
defines the underlying multiplicative monoid of the tilt of $F$.
We have thus established the following canonical bijection:

\begin{proposition}\label{propflat}
Let $F$ be an arbitrary field endowed with the coarse topology. The map $\chi\mapsto  \chi^\flat$  from points of the $\Fone$-curve $\spzf$ at the prime $p$ with values in $F$ to the tilt $F^\flat$: 
\begin{equation}\label{flatmap}
(\chi^\flat)_n:=\chi\left(\frac{1}{p^n}\right)
\end{equation}
is a canonical  bijection $\beta$ with the tilt $F^\flat$
\begin{equation}\label{flatmap}
\beta: \Pt_p(\spzf,F)\to F^\flat.
\end{equation}
\end{proposition}

This proposition establishes that at the purely algebraic level, the $\Fone$-stalk at $p$ of $\mathcal F$ acts as a universal tilting functor; evaluating it over any field $F$ automatically extracts its tilt $F^\flat$. 

It is very interesting to see how this moduli space collapses when we endow the field $F$ with the \emph{discrete} topology. In this case, the continuity condition imposed on the points of the $\Fone$-curve acts as a rigid arithmetic filter. In the natural topology of the stalk, the sequence $p^n$ converges additively to $0$ as $n \to \infty$. For the homomorphism $\chi$ to be continuous, its multiplicative image must converge to $\chi(0) = 1$:
\[
\lim_{n \to \infty} \chi(p^n) = 1 \in F.
\]
Because $F$ is endowed with the discrete topology, a convergent sequence must be eventually constant. Thus, there exists an integer $N$ such that for all $n \ge N$, $\chi(p^n) = 1$. 

Since $\chi(p^n) = \chi(1)^{p^n} = x_0^{p^n}$, this implies that $x_0^{p^n} = 1$, meaning $x_0:=\chi(1)$ is a $p$-power root of unity. By the same logic, $\chi(1/p^k)^{p^{n+k}} = \chi(p^n) = 1$, forcing every $x_k$ to be a $p$-power root of unity. Consequently, one sees that the image of $\chi$ lands entirely within $\mu_{p^\infty}(F)$, the group of roots of unity in $F$ of $p$-power order. 

Because the target $\mu_{p^\infty}(F)$ is a group, the monoid homomorphism $\chi : \Z[1/p]_+ \to \mu_{p^\infty}(F)$ automatically and uniquely extends to a group homomorphism from the Grothendieck group $H_p = \Z[1/p]$. We thus obtain the following precise classification:

\begin{proposition}
Let $F$ be a field endowed with the discrete topology. The  points of the $\Fone$-curve $\spzf$ over the prime $p$ with values in $F$ correspond exactly to group homomorphisms:
\[
 \Pt_p(\spzf,F)\cong\Hom_{\Grp}(H_p,\mu_{p^\infty}(F)).
\]
\end{proposition}

This reveals a  topological trichotomy governing the $\Fone$-stalk at $p$:
\begin{enumerate}
    \item \emph{Coarse Topology:} The moduli space is the full tilt $F^\flat$.
    \item \emph{Discrete Topology:}  The moduli space isolates the $p$-power roots of unity $\mu_{p^\infty}(F)$, rigidly collapsing the continuous degrees of freedom of the stalk into discrete torsion.\vspace{.03in}

    In the following part we shall see that:
    
    \item \emph{Analytic Topology:} When $F$ is a perfectoid field endowed with its analytic valuation topology, the continuous local points rigidly carve out the moduli space of untilts---thereby generating the closed points of the Fargues--Fontaine curve.
    \end{enumerate}

    \subsection{Points of $\spzf$ in a perfectoid field}

    Let $C$ be a perfectoid field of residue characteristic $p$, satisfying the standard axioms (\cf \cite[Lecture 2, Def.~1]{Lurie}). That is, $C$  is complete with respect to a nonarchimedean absolute value $|\cdot|_C$, its residue field has characteristic $p$, meaning $|p|_C < 1$, and  the Frobenius is surjective modulo $p$, and the value group is non-discrete. \vspace{.03in}

We first describe the  analytic moduli space $\Pt_p(\spzf,C)$ at the prime $p$.   
 We know, from Proposition \ref{classpla} that given the ordered group $H_p=\Z[1/p]$, the set of places $\Sigma_{H_p}$ includes the archimedean place $\infty$ (corresponding to the completion $\R$) alongside the non-archimedean place $p$. A local morphism must be continuous with respect to one of the associated  topologies. 
However, $C$ is a perfectoid field equipped with a non-archimedean absolute value, thus its underlying topological space is totally disconnected. Because the additive group $\R$ is connected, the continuous image of $\R$ in $C^\times$ must be connected. The only connected subsets of a totally disconnected space are singletons. Since the map $\R\to C^\times$ is a group homomorphism, this singleton must be the identity $1 \in C^\times$. Thus, there are no non-trivial continuous homomorphisms from $\R$ to $C^\times$. Consequently,  we may restrict our attention  to study the continuity for the $p$-adic topology.
 The natural symmetries of the moduli space $\Pt_p(\spzf,C)$ are governed by the group $\Q_p^\times \simeq p^\Z \times \Z_p^\times$ acting by automorphisms of the $p$-adic completion $\Q_p$ of $H_p$. 

 Given a perfectoid field $C$ we let 
\begin{equation}\label{perfideal}
\mathfrak{m}_{C}:=\{x\in C\mid \vert x\vert<1\}
\end{equation}
be the unique maximal ideal of the valuation ring  $\mathcal{O}_C:=\{x\in C\mid \vert x\vert \leq 1\}$.

The multiplicative monoid $C^\flat$ becomes a perfectoid field of characteristic $p$ when endowed with the addition
\begin{equation}\label{perfadd}
(x+y)_n:=\lim _{m \rightarrow \infty}\left(x_{n+m}+y_{n+m}\right)^{p^m}
\end{equation}
and the norm given by 
\begin{equation}\label{perfnorm}
\vert x\vert_{C^\flat}:=\vert x_0\vert_{C}.
\end{equation}

\begin{theorem}\label{main}
Let $C$ be a perfectoid field of residue characteristic $p$.
\begin{enumerate}
\item[(i)]~The  points of the $\Fone$-curve $\spzf$ at the prime $p$ with values in $C$ correspond precisely to group homomorphisms:
\[
 \Pt_p(\spzf,C)\simeq\Hom_{\Grp}(H_p,1 + \mathfrak{m}_{C})
\]
\item[(ii)]~The evaluation map $\beta: \Pt_p(\spzf,C)\to C^\flat$ $\beta(\chi)=\chi^\flat$ induces a canonical bijection  $$\beta:\Pt_p(\spzf,C)  \stackrel{\sim}{\to} D := 1 + \mathfrak{m}_{C^\flat}$$ 
\item[(iii)]~The bijection $\beta$ in $(ii)$ is $\Q_p^\times$-equivariant.
\end{enumerate}
\end{theorem}

\begin{proof} 
$(i)$~A  monoid homomorphism $\chi : \Z[1/p]_+ \to C$ is uniquely determined by the sequence $x_n = \chi(1/p^n) \in C$, which satisfies $x_n^p = x_{n-1}$. For a homomorphism $\chi : \Z[1/p]_+ \to C$ to be continuous, the topological convergence $p^n \to 0$ in the stalk must translate to the multiplicative convergence $\chi(p^n) \to 1$ in $C$. Letting $x_0 = \chi(1)$, this requires:
\[
\lim_{n \to \infty} x_0^{p^n} = 1 \in C.
\]
This analytic condition imposes a strong restriction on the choice of $x_0$. We can classify the behavior of $x_0^{p^n}$ by the  norm:
\begin{itemize}
    \item If $|x_0|_p < 1$, then $x_0^{p^n} \to 0$.
    \item If $|x_0|_p > 1$, then $|x_0^{p^n}|_p \to \infty$.
    \item If $|x_0|_p = 1$ but $x_0 \notin 1 + \mathfrak{m}_{C}$ (where $\mathfrak{m}_{C}$ is the maximal ideal \eqref{perfideal}), its reduction $\alpha$ in the residue field $
k=\mathcal{O}_C / \mathfrak{m}_C
$ is not $1$. Because the Frobenius $a \mapsto a^p$ is an automorphism of $k$, the elements $\alpha^{p^n}\in k$ are all  $\neq 1$, meaning $|x_0^{p^n} - 1|= 1, \forall n$.
    \item If $x_0 \in 1 + \mathfrak{m}_{C}$ (\ie $c=|x_0 - 1| < 1$), the convergence $x_0^{p^n}\to 1$ follows since the iterates of the $p$-th power map converge uniformly to $1$ in the closed disk $\{z\mid \vert z-1\vert\leq c$. Writing $z = 1+y$ with $|y| < c$, the binomial expansion yields $$|z^p - 1| \le \max(|py|, |y^p|)\leq \max(\vert p\vert ,c^{p-1} )\vert z-1\vert. $$ 
    Iterating this contraction guarantees $x_0^{p^n} \to 1$. Note that because C is a perfectoid field of residue characteristic p, we have |p| < 1. 
\end{itemize}
Therefore, the continuity condition $\chi(p^n) \to 1$ is satisfied if and only if $x_0$ belongs to the group of principal units $1 + \mathfrak{m}_{C}$. Because the sequence defining the point satisfies $x_n^p = x_{n-1}$, every $x_n$ must also lie in $1 + \mathfrak{m}_{C}$. It then follows that the homomorphism $\chi : \Z[1/p]_+ \to C$ is continuous for the induced $p$-adic topology on $H_p=\Z[1/p]\subset \Q_p$. Indeed, for any integer $a>0$ relatively prime to $p$ one has 
$$
\chi(ap^n)=\chi(p^n)^a,
$$
and the binomial formula shows that 
$$
\vert z-1\vert <1\implies \vert z^a-1\vert \leq \vert z-1\vert.
$$

$(ii)$~To verify that $x^\flat$ lies in the open unit disk $1 + \mathfrak{m}_{C^\flat}$, we must show that $|x^\flat - 1^\flat|_{C^\flat} < 1$. 

In the tilt $C^\flat$, addition is defined via the limit \eqref{perfadd}. Let $y^\flat = x^\flat - 1^\flat$; its $0$-th component in $C$ is given by the convergent sequence\footnote{This still holds for $p=2$, and in that case one has $2=0$ in $C^\flat$ and $\lim _{n \rightarrow \infty}\left(x_n-1\right)^{2^n}=\lim _{n \rightarrow \infty}\left(x_n+1\right)^{2^n}$ in $C$}:
\[
y_0 = \lim_{n \to \infty} (x_n - 1)^{p^n}.
\]
Because $x_n \in 1 + \mathfrak{m}_C$, we have $|x_n - 1|_C < 1$, and therefore $|(x_n - 1)^{p^n}|_C < 1$ for all $n$. In any non-archimedean field, if a sequence converges to a limit $y_0$, its sequence of norms must eventually be stationary (if $y_0 \neq 0$) or converge to $0$ (if $y_0 = 0$). In either case, since every term in the sequence has norm strictly less than $1$, the norm of the limit must satisfy $|y_0|_C < 1$. Therefore, $|y^\flat|_{C^\flat} = |y_0|_C < 1$, proving that $x^\flat$ lies precisely in the open unit disk $D = 1 + \mathfrak{m}_{C^\flat}$. Since $C$ is perfectoid, every such sequence in $C$ lifts to a valid continuous homomorphism, proving the bijection.

$(iii)$~To verify $\Q_p^\times$-equivariance, we define the action on the morphisms. The action of $p \in p^\Z$ is given by pre-composition with multiplication by $p$: $(p \cdot \chi)(a) = \chi(pa)$. On the sequence, this shifts the evaluation: $x_n \mapsto \chi(p/p^n) = \chi(1/p^{n-1}) = x_{n-1} = x_n^p$. This exactly matches the Frobenius action $\varphi(x^\flat) = (x^\flat)^p$ on $D$. 
Similarly, for $u \in \Z_p^\times$, the action is $\chi^u(a) = \chi(a)^u$. On the sequence, this yields $(x_n^u)$, which exactly matches the exponentiation action $(x^\flat)^u$ on the $\Z_p$-module $D$.
On the open unit disk $D = 1 + \mathfrak{m}_{C^\flat}$ of the tilt, the group $\Q_p^\times$ acts as follows: $p \in p^\Z$ acts via the Frobenius $\varphi(x^\flat) = (x^\flat)^p$, and $u \in \Z_p^\times$ acts via the natural exponentiation $x^\flat \mapsto (x^\flat)^u$, which is well-defined because $D$ is a $\Z_p$-module.
\end{proof}

Theorem \ref{main} shows a profound feature of absolute $\Fone$-geometry: the moduli space of points of $\spzf$ does not require the test field $C$ to be of characteristic $p$. Let $C$ be any  perfectoid field, regardless of its characteristic (for instance, $C = \C_p$ in characteristic $0$). The algebraic structure of the stalk $\Z[1/p]_+$ automatically forces the extraction of $p$-th roots, while the  topology forces the image into the principal units. Consequently, evaluating the  points of the $\Fone$-curve $\spzb$ over $C$ canonically yields the open unit disk of its tilt, $1 + \mathfrak{m}_{C^\flat}$. 

If $C$ is of characteristic $p$, then $C^\flat = C$. However, if $C$ is of characteristic $0$, the $\Fone$-stalk acts as a universal tilting functor, automatically passing to $C^\flat$. The natural symmetries of the stalk at $p$ correspond to the action of $\Q_p^\times \simeq p^\Z \times \Z_p^\times$, where $p^\Z$ acts via the Frobenius and $\Z_p^\times$ acts as the Galois gauge group. Removing the trivial point $x_0=1$ and quotienting by this $\Q_p^\times$-action yields exactly the classical points of the Fargues--Fontaine curve $\mathcal{X}_{C^\flat, \Q_p}$.\vspace{.03in} 

Thus, the absolute $\Fone$-curve $\spzf$ is intrinsically compatible with the tilting correspondence. It serves as a universal geometric receptacle for perfectoid fields of arbitrary characteristic, canonically extracting their tilt and recovering the moduli space of their characteristic-zero untilts.

\subsection{Points of $\spzf$ in a perfectoid field of characteristic $p$}

Next, we specialize Theorem  \ref{main} to 
the case of a perfectoid field $C$ of characteristic $p$. 

\begin{corollary}\label{basic} 
Let $H_p = \Z[1/p]$, and let $F$ be a perfectoid field 
of characteristic $p$.
\begin{enumerate}
    \item[(i)] The  map: 
    \[
\Hom(\Fone[T^{H_p^+}],HF)\longrightarrow F\qquad \rho \longmapsto \rho(T)
    \]
    determines  a canonical bijection. 
    For each $x \in F$, the corresponding morphism 
    $\psi_x: \Fone[T^{H_p^+}]\to HF$ is uniquely determined by setting:
    \[
    \psi_x(T^{a/p^k}) = x^{a/p^k}.
    \]

    \item[(ii)] The morphism $\rho = \psi_x$ is local 
    at $p$ (\cf Definition~\ref{defnloc}) 
    if and only if: 
    \[
    |x - 1| < 1.
    \]
\end{enumerate}
\end{corollary}

\begin{proof}
$(i)$~Follows from Proposition \ref{propflat} and the identity $F=F^\flat$ coming from the existence and uniqueness of $p$-th roots in $F$.\newline
$(ii)$~Follows from Theorem \ref{main}.
\end{proof}

Corollary~\ref{basic} exhibits a relevant geometric phenomenon. 
The full space of absolute $F$-points of the stalk
$
\Fone[T^{H_p^+}]
$
identifies canonically with the affine line over $F$. 
The locality condition, however, imposes a strong rigidity constraint: compatibility with the $p$-adic topology forces the image of the geometric point to lie in the open unit disk
$
1+\mathfrak m_F$, 
which characterizes precisely the Fargues--Fontaine locus.

\subsection{Cyclotomic embeddings}

The next part reinterprets the  geometric local points of $\spzf$ in the language of perfectoid untilts. Local $\Fone$-morphisms get identified with cyclotomic embeddings into characteristic-zero untilts of a perfectoid field $F$.\vspace{.03in}

We start by recalling the definition 
of the $p$-adic cyclotomic field.

\begin{definition}
Let $\mu_{p^\infty}$ denote the group of all 
$p$-power roots of unity in an algebraic closure 
$\overline{\Q}_p$. 
The cyclotomic extension 
$\Q_p(\mu_{p^\infty})$ 
is obtained by adjoining all such roots to $\Q_p$. 
Its $p$-adic completion is a perfectoid field 
of characteristic zero, which we denote by 
$\Q_p^{\mathrm{cyc}}$.
\end{definition}

Let $F$ be an algebraically closed perfectoid field 
of characteristic $p$. 
Recall (\cite{Scholze}) that an \emph{untilt} of $F$ 
is a pair $(C,\iota)$, where 
$\iota: F \xrightarrow{\sim} C^\flat$ 
is an isomorphism of fields that identifies 
the valuation ring $\mathcal O_F$ 
with the valuation ring $\mathcal O_{C^\flat}$.\vspace{.03in}

 Corollary~\ref{basic} $(ii)$ states that the set of \emph{local morphisms} 
from the stalk at $p$ of the structure sheaf of $\spzf$
to $HF$ coincides with the open unit disk inside $F$:
\[
\operatorname{Hom}_{\mathrm{loc}}
(\Fone[T^{H_p^+}], HF)
=
\{ \psi_x \mid x\in F,\ |x-1|<1 \}.
\]

We now translate Proposition~2 in 
 Lecture~8 of \cite{Lurie} 
into the language of absolute geometry. 
This result shows that a local $\Fone$-morphism 
is not merely an abstract point, but corresponds 
precisely to the algebraic data of a cyclotomic 
embedding.

\begin{proposition}\label{lurieprop2}
Let 
$\operatorname{Hom}_{\mathrm{loc}}^\circ
(\Fone[T^{H_p^+}], HF)$ 
denote the set of non-trivial \footnote{those for which 
$x=\psi_x(T)\neq1$, using the notation of Corollary~\ref{basic}} local morphisms. 
There is a canonical, $\Z_p^\times$-equivariant 
bijection between:
\begin{enumerate}
    \item The set of non-trivial local morphisms: 
    $
\operatorname{Hom}_{\mathrm{loc}}^\circ
    (\Fone[T^{H_p^+}], HF)
    $\vspace{.02in}

    and\vspace{.02in}
    
    \item The set of triples $(C,\iota,u)$, where 
    $(C,\iota)$ is a perfectoid untilt of $F$, 
    and 
    \[
    u:\Q_p^{\mathrm{cyc}}\hookrightarrow C
    \]
    is a continuous embedding of topological fields.
\end{enumerate}
\end{proposition}

\begin{proof}
By Corollary~\ref{basic} $(ii)$, a non-trivial local morphism 
$\psi_x$ is uniquely determined by an element 
of the punctured open unit disk
\[
D=\{x\in F\mid 0<|x-1|<1\}.
\]
The statement then follows directly from \cite[Proposition~2 Lecture~8]{Lurie}, 
where one identifies such elements with the data 
of perfectoid untilts equipped with a 
cyclotomic embedding.
\end{proof}

We briefly recall the construction (described in details in \cite{Lurie}) 
of the bijection between the disk $D\subset F$ and the triples 
$(C,\iota,u)$. 

Let 
\[
\epsilon\in(\Q_p^{\mathrm{cyc}})^\flat
\]
be the element defined by the compatible system 
of $p$-power roots of unity:
\[
\epsilon
:=
\left(
1,\zeta_p,\zeta_{p^2},\zeta_{p^3},\ldots
\right).
\]

Given a triple $(C,\iota,u)$, consider the element
\[
\epsilon'
=
u^\flat(\epsilon)
\in C^\flat.
\]
It satisfies
\[
0<|\epsilon'-1|_{C^\flat}<1,
\]
and therefore
\[
x
:=
\iota^{-1}(\epsilon')
\]
belongs to the disk $D$.
Conversely, let 
$\mathbf A_{\mathrm{inf}}$
denote the ring of Witt vectors of the valuation ring 
$\mathcal O_F$ of $F$. 
Given an element $x\in D$, define
\begin{equation}\label{2perf}
\xi
=
1+[x^{1/p}]
+\cdots+
[x^{(p-1)/p}]
\in \mathbf A_{\mathrm{inf}},
\end{equation}
where the brackets denote the Teichm\"uller lift. 
One shows that $\xi$ is a distinguished element (\cf \cite[Definition 15, Lecture 3]{Lurie}) of 
$\mathbf A_{\mathrm{inf}}$, which allows the construction 
of a triple $(C,\iota,u)$. 
The perfectoid field $C$ is obtained as the field of fractions of the quotient $
\mathbf A_{\mathrm{inf}}/(\xi)$, 
and the embedding $
u:\Q_p^{\mathrm{cyc}}
\hookrightarrow C$ 
is determined by the compatible system $
[x^{1/p^n}]
\in
\mathbf A_{\mathrm{inf}}/(\xi)$.
We refer to \cite[Lecture~8]{Lurie} 
for the detailed proof that these constructions 
establish a bijection:
\[
\{(C,\iota,u)\}
\stackbin[\sim]{\iota^{-1}u^\flat(\epsilon)}{\xrightarrow{\hspace{1cm}}}
\left\{
x\in F
\;\middle|\;
0<|x-1|_F<1
\right\}.
\]

\begin{remark}\label{2perf1}
In our framework, the choice of $x \in 1 + \mathfrak{m}_F$ completely dictates the characteristic of the resulting untilt. In $p$-adic Hodge theory, the untilt $K$ corresponding to $x$ is constructed as the field of fractions of the quotient $\mathcal{O}_K = \mathbf A_{\text{inf}} / (\xi)$, where $\xi$ is the  distinguished element of degree 1 given by \eqref{2perf}. For $x \neq 1$, the ideal $(\xi)$ does not contain $p$, and the resulting untilt $K$ has characteristic zero. 

However, if one makes the trivial choice $x = 1$ (corresponding, in our case, to the trivial $\Fone$-morphism $1^h = 1$), the polynomial formula \eqref{2perf} for the distinguished element evaluates exactly to:
\[
\xi = \underbrace{1 + 1 + 1 + \dots + 1}_{p \text{ times}} = p.
\]
Consequently, the ideal generated by $\xi$ is simply $p \mathbf A_{\text{inf}}$. The resulting untilt is the field of fractions of the quotient $\mathbf A_{\text{inf}} / (p) \simeq \mathcal{O}_F$, which is exactly the trivial untilt $F$ of characteristic $p$. Thus, the algebraic specialization of the distinguished element $\xi \to p$ as $x \to 1$ rigorously demonstrates why the trivial choice $x=1$ corresponds to the closed point $p \in \Spec(\Z)$. The Fargues--Fontaine curve $\mathcal{X}_{F, \Q_p}$, which lives over $\Q_p$, naturally excludes this trivial point, and so its set of closed points corresponds exactly to the restriction to the punctured disk $(1 + \mathfrak{m}_F) \setminus \{1\}$.
\end{remark}

\subsection{Symmetries and the Fargues-Fontaine curve}\label{ffcurve}

In the next part, we study the natural symmetries of the completed stalk. Quotienting the moduli space of local points by these symmetries recovers first the untilts of the perfectoid field $F$, and then, after Frobenius quotient, the closed points of the Fargues--Fontaine curve.\vspace{.05in}

Let $F$ be an algebraically closed perfectoid field of characteristic $p$. 
The space of  local morphisms 
\[
\operatorname{Hom}_{\mathrm{loc}}(\Fone[T^{H_p^+}], HF)
\]
naturally carries an action of the continuous automorphism group of the completed stalk.

Because the stalk is generated by the group $H_p=\Z[1/p]$, its local completion at $p$ is $\Q_p$. 
As a topological additive group, the automorphism group of $\Q_p$ has the following description:
\[
\operatorname{Aut}(\Q_p) 
\cong 
\Q_p^\times 
\cong 
p^\Z \times \Z_p^\times.
\]
For any $a\in\Q_p^\times$ and any local morphism 
$\psi_x\in\Hom_{\mathrm{loc}}(\Fone[T^{H_p^+}],HF)$, 
the action is given by
\[
\psi_x \longmapsto \psi_{x^a}.
\]

In view of the bijection established in Proposition~\ref{lurieprop2}, 
the action of an element 
$u\in\Z_p^\times$ 
corresponds exactly to the Galois action of
\[
\operatorname{Gal}(\Q_p^{cyc}/\Q_p)
\cong 
\Z_p^\times
\]
on the cyclotomic embeddings, that modifies the choice of the compatible system of primitive roots of unity.\vspace{.05in}

We may now interpret Corollaries~3 and~4 of Lecture~8 in \cite{Lurie}
as describing the geometric quotients of this moduli space under these intrinsic symmetries of the topos.
Let denote by
$$
\mathcal{M}_{loc}:=\operatorname{Hom}_{\mathrm{loc}}(\Fone[T^{H_p^+}], HF)\setminus\{1\}
$$
the space of  non-trivial local morphisms.
\begin{corollary}\label{luriecor3}
There is a canonical bijection between:
\begin{enumerate}
    \item The quotient space 
    $\mathcal{M}_{loc}/ \Z_p^\times$ 
    of non-trivial local morphisms modulo the compact automorphism group of the completed stalk.
    
    \item The set of isomorphism classes of perfectoid untilts $C$ (of characteristic zero) of $F$ that contain all $p$-power roots of unity.
\end{enumerate}
\end{corollary}

\begin{proof}
By Proposition~\ref{lurieprop2}, the space 
$\mathcal{M}_{loc}$ 
is in bijection with triples 
$(C,\iota,u)$, 
where $(C,\iota)$ is an untilt of $F$ and 
$u:\Q_p^{cyc}\hookrightarrow C$ 
is a continuous embedding. This embedding $u$ is uniquely determined by the choice of a compatible system of primitive $p$-power roots of unity, which algebraically corresponds to the choice of a generator of the Tate module $\Z_p(1)$ inside $C$.
The action of an element 
$v\in\Z_p^\times$ 
on $\mathcal{M}_{loc}$ 
is given by 
$\psi_x\mapsto\psi_{x^v}$. 
On the level of the Tate module, this action corresponds exactly to changing the choice of generator. 

Consequently, passing to the quotient by $\Z_p^\times$ removes the dependence on the specific embedding $u$ and retains only the isomorphism class of the untilt $C$, provided that such an embedding exists (\ie that $C$ contains the $p$-power roots of unity).
\end{proof}

To recover the closed points of the Fargues–Fontaine curve $\mathcal{X}_{F, \Q_p}$, we must quotient by the full automorphism group $\Q_p^\times$. 
The remaining factor is the discrete subgroup $p^\Z$, generated by multiplication by $p$. 
On the space of local morphisms, this subgroup acts by
\[
\psi_x \longmapsto \psi_{x^p},
\]
which coincides with the geometric Frobenius map $\varphi$.

\begin{corollary}\label{luriecor4}
There is a canonical bijection between:
\begin{enumerate}
    \item The quotient space 
    $\mathcal{M}_{loc}/\Q_p^\times$ 
    of non-trivial local morphisms modulo the full automorphism group of the completed stalk.
    
    \item The closed points of the Fargues–Fontaine curve 
    $\mathcal{X}_{F,\Q_p}$ over $F$.
\end{enumerate}
\end{corollary}

\begin{proof}
By definition, the closed points of the Fargues–Fontaine curve $\mathcal{X}_{F,\Q_p}$ over $F$ are in canonical bijection with the Frobenius orbits of the characteristic-zero untilts of $F$. 

By Corollary~\ref{luriecor3}, the quotient 
$\mathcal{M}_{loc}/\Z_p^\times$ 
yields the set of untilts $C$. 
The residual action of 
$p\in p^\Z\subset \Q_p^\times$ 
maps an untilt $C$ to its Frobenius shift $C^\varphi$, that is, the untilt whose tilt identification is twisted by the Frobenius of $F$. 

Therefore, taking the full quotient by 
$\Q_p^\times \cong p^\Z \times \Z_p^\times$ 
is equivalent to first passing to the set of untilts and then quotienting by the Frobenius equivalence relation. 
This produces precisely the set of closed points of $\mathcal{X}$.
\end{proof}

The preceding sequence of corollaries yields a striking conceptual unification. 
The set of closed points of the Fargues–Fontaine curve  is obtained  from the open unit disk
\[
Y=\{\,0<|x-1|<1\,\}
\]
and quotienting by the Frobenius action $x\mapsto x^p$. 
Traditionally, the need to quotient first by $\Z_p^\times$ 
(the cyclotomic symmetry) and then by $p^\Z$ 
(the Frobenius action) appears as a collection of distinct arithmetic steps.

On the other hand, from the perspective of the $\Fone$-geometry,  these two operations arise as components of a single canonical construction. 
The closed points of the Fargues–Fontaine curve $\mathcal{X}_{F, \Q_p}$ may be viewed as the moduli space of local $F$-points  at $p$ of the  $\Fone$-arithmetic curve $\spzf=(\spz, \mathcal F)$,
taken modulo the natural topological symmetries 
$\operatorname{Aut}(\Q_p)$ of the completed stalk. 
In this way, the absolute $\Fone$-geometry  intrinsically encodes the construction of $\mathcal{X}_{F,\Q_p}$.

\subsection{Scholze's heuristic}\label{heuristic}

In the foundational philosophy of $p$-adic Hodge theory, P. Scholze proposed the following heuristic idea (\cf \cite[Lecture 1]{Lurie}): \vspace{.05in}

\emph{For an algebraically closed perfectoid field $F$ of characteristic $p$, the set of untilts of $F$ modulo isomorphism serves as a good replacement for the set of $F$-valued points of $\spz$.}\vspace{.05in} 

In this section we demonstrate that the $\Fone$-curve $\spzf$ gives a satisfactory answer to this requirement.  The stalks are given as follows:
\begin{itemize}
    \item $\mathcal{F}_p = \Fone[T^{\Z[1/p]_+}]$ at a closed point $p$ (for any prime $p$).
    \item $\mathcal{F}_\eta = \Fone$ at the generic point $\eta\in \spz$.
\end{itemize}

An $F$-point of $\spzf$ is given by the choice of a  point $x \in \Spec(\Z)$ together with a non-trivial local morphism of $\Fone$-algebras:   $\mathcal{F}_x\to HF$.

\begin{theorem}\label{scholze_heuristic}
Let $F$ be an algebraically closed perfectoid field of positive characteristic $p$. The moduli space of geometric $F$-points of $\spzf$, modulo the canonical symmetries of the stalks, decomposes over the closed points of $\Spec(\Z)$ as follows:
\begin{enumerate}
    \item At any prime different from $p$, the fiber rigidly collapses to a single trivial point.
    \item At the prime $p$, the fiber is canonically in bijection with the space of all untilts of $F$ modulo the Frobenius.
\end{enumerate}
\end{theorem}

\begin{proof} We study the local $F$-points of $\spzf=(\spz,\mathcal F)$ at each  point $x \in \spz$.\vspace{.03in}

\emph{Case 1: The closed point is $x = q$, for a prime $q \neq p$.} 

A morphism $\rho: \mathcal{F}_q \to HF$ corresponds to an element $y = \rho(T) \in F$. Because $F$ is algebraically closed, $y$ has $q^n$-th roots, so there exist algebraic morphisms $\rho$ such that $y = \rho(T)$. 
However, for the morphism to be \emph{local} at $q$, it must be continuous with respect to the $q$-adic topology of the stalk. This requires:
\[ 
\lim_{n \to \infty} \rho(T^{q^n}) = \lim_{n \to \infty} y^{q^n} = 1. 
\]
Because $F$ has characteristic $p\neq q$, the integer $q$ is a unit in the valuation ring $\mathcal{O}_F$. Therefore, 
$$
z\in F, \ |z - 1| < 1\implies |z^q - 1| = |z - 1|,
$$
since with $\epsilon=z-1$ one has
$$
z^q - 1=q \epsilon +\sum_{k>1} \binom{q}{k}\epsilon^k.
$$
Thus, if $\lim_{n \to \infty} y^{q^n} = 1$ the sequence $|y^{q^n} - 1|$ is constant equal to $0$ for $n$ large, and $y$ is a $q^n$ root of unity. 
Hence, the only local morphisms at $q \neq p$ are the morphisms $y = \zeta_{q^n}$. Because $F$ is an algebraically closed field of characteristic $p$, it contains the algebraic closure $\overline{\mathbf F}_p$ of the prime field. Any non-zero element $z \in \overline{\mathbf F}_p^\times$ belongs to a finite field, meaning $z^m = 1$ for some integer $m$. Taking the absolute value in $F$ yields $|z|^m = 1$, which implies $|z| = 1$. Consequently, for any two distinct elements $a, b \in \overline{\mathbf F}_p$, their difference $a - b$ is a non-zero element of $\overline{\mathbf F}_p$, and thus $|a - b| = 1$. This proves that the induced metric on $\overline{\mathbf F}_p$ is the trivial discrete metric. Because $q \neq p$, the $q$-power roots of unity $\mu_{q^\infty}$ are entirely contained in $\overline{\mathbf F}_p$, and therefore inherit this discrete topology. By fixing a choice of primitive roots, we can identify $\mu_{q^\infty}\subset \overline{\mathbf F}_p$ with the quotient group $\Q_q/\Z_q$. 
Given any homomorphism $\phi \in \Hom(\Z[1/q], \mu_{q^\infty})$, we can evaluate it at $1/q^n$ to obtain an element $x_n = \phi(1/q^n) \in \Q_q/\Z_q$. Multiplying by $q^n$ yields a well-defined element:
\[ z_n = q^n x_n \in \Q_q / q^n \Z_q. 
\]
Because $\phi$ is a homomorphism, we have $q x_{n+1} = x_n$, which implies:
\[ 
z_{n+1} = q^{n+1} x_{n+1} = q^n (q x_{n+1}) = q^n x_n = z_n \pmod{q^n \Z_q}. 
\]
Thus, the sequence $(z_n)$ defines a unique element in the inverse limit $\varprojlim (\Q_q / q^n \Z_q)$, which is exactly the additive group $\Q_q$. This provides an isomorphism between $\Hom(\Z[1/q], \mu_{q^\infty})$ and $\Q_q$.

Thus, the space of non-trivial local morphisms $\mathcal{M}_q=\operatorname{Hom}_{loc}(\Fone[T^{H_q^+}], HF)\setminus\{1\}$ is parameterized exactly by the non-zero elements $\Q_q \setminus \{0\} = \Q_q^\times$. To obtain the geometric points, we must quotient by the canonical symmetries of the stalk, which is $\operatorname{Aut}(\Q_q) \cong \Q_q^\times$. The action of $\Q_q^\times$ on $\mathcal{M}_q$ corresponds to the multiplicative action of $\Q_q^\times$ on itself, which is strictly transitive. Therefore, the quotient $\mathcal{M}_q / \Q_q^\times$ collapses to a single trivial point and so there is no continuous geometry at $q \neq p$.\vspace{.03in}

\emph{Case 2: The generic point $x = \eta$.} 

The stalk of $\mathcal O_{\Fone}$ at $x=\eta$ is $\Fone$, corresponding to the group $H = \{0\}$. By Proposition \ref{classpla}, the set of places $\Sigma_H$ is empty. There is no dense embedding into any local field, meaning the locality condition cannot be fulfilled. Thus, the generic point yields no continuous local geometry.\vspace{.03in}

\emph{Case 3: The closed point $x = p$.} 

The stalk is here $\mathcal{F}_p = \Fone[T^{\Z[1/p]_+}]$. As established in Theorem \ref{main}, the locality condition $\lim_{n \to \infty} y^{p^n} = 1$ is highly non-trivial because the Frobenius map $z \mapsto z^p$ is strictly contracting on $1 + \mathfrak{m}_F$ (since $|y^p - 1| = |y - 1|^p$). 
Thus, the non-trivial local $F$-points at $p$ form exactly the open unit disk $0 < |y - 1| < 1$. 

By Corollaries \ref{luriecor3} and \ref{luriecor4}, quotienting this space of local morphisms by the canonical  symmetries of the stalk ($\operatorname{Aut}(\Q_p) \cong \Q_p^\times$) yields exactly the untilts of $F$ modulo Frobenius, which are the closed points of the Fargues-Fontaine curve. Moreover the trivial local $F$-point $y=1$ corresponds by Remark \ref{2perf1} to the trivial untilt of $F$, and this corresponds precisely to the inclusion of this trivial untilt in the Scholze heuristic.
\end{proof}

Theorem~\ref{scholze_heuristic} provides a  geometric validation of Scholze's heuristic. The \emph{$\Fone$-arithmetic curve} $\spzf$
acts as an absolute geometric sieve. When probed by a field of characteristic $p$, the ultrametric dynamics and the canonical topos symmetries rigidly collapse the geometry at all primes $q \neq p$, reducing each of their fibers to a single, trivial point. On the other hand, the geometry at the prime $p$ completely opens up: the $\Fone$-structure automatically generates the cyclotomic embeddings, and the fiber reveals the infinite-dimensional moduli space related to the set of closed points of the Fargues-Fontaine curve (including the trivial untilt as required by Scholze's heuristic).

\subsection{The universal tilting functor and the geometric sieve}

Let $C$ be a perfectoid field of residue characteristic $p$.
Theorem~\ref{main} provides a far-reaching extension of Scholze's heuristic principle.

The prime $p$ plays a distinguished role, giving rise to the full moduli space of untilts. We now contrast this behavior with that of the primes $\ell\neq p$, where the geometry collapses. The resulting dichotomy yields the global geometric sieve. \vspace{.03in}

Throughout the remainder of this part, we assume that $C$ is algebraically closed.\vspace{.05in}

\emph{Case 1: The mismatched primes ($\ell = q \neq p$).} \\
Because $q$ is coprime to $p$, we have $|q|_C = 1$. For any element $x = 1+y \in 1 + \mathfrak{m}_C$, the binomial expansion yields $x^q = 1 + qy + \dots + y^q$. Since $|q|_C=1$ and $|y|_C<1$, the strictly dominant term is $qy$, meaning $|x^q - 1|_C = |x - 1|_C$. 

Thus, the $q$-th power map is an \emph{isometry} on the principal units. An isometric sequence $x_0^{q^n}$ cannot converge to $1$ unless it is eventually exactly equal to $1$. Therefore, continuity forces $x_0 \in \mu_{q^\infty}(C)$. The analytic topology at $q \neq p$ rigidly collapses, yielding the exact same moduli space as if $C$ were endowed with the discrete topology:
\[
\operatorname{Hom}_{\text{cont}}(\Z[1/q]_+, C) \simeq \operatorname{Hom}_{\text{group}}(H_q, \mu_{q^\infty}(C))
\]

\emph{Case 2: The matched prime ($\ell = p$).} \\
By Theorem \ref{main}, the continuous geometric points at $p$ form the open unit disk $D = 1 + \mathfrak{m}_{C^\flat}$. Removing the trivial point $x_0=1$ (which corresponds to the characteristic $p$ untilt $C^\flat$) and quotienting by the equivariant $\Q_p^\times$-action yields exactly the set of closed points of the Fargues--Fontaine curve $\mathcal{X}_{C^\flat, \Q_p}$.\vspace{.03in}

We have thus established the following   theorem, which serves as the absolute realization of Scholze's paradigm.

\begin{theorem}
Let $C$ be an algebraically closed perfectoid field of residue characteristic $p$. Evaluating the  $\Fone$-curve $\spzf$ over $C$ acts as a strict geometric sieve:
\begin{enumerate}
    \item At every prime $\ell \neq p$, the ultrametric mismatch forces the analytic topology to behave discretely, collapsing the moduli space to a single $\Q_\ell^\times$-orbit.
    \item At the prime $p$, the $\Fone$-stalk acts as a universal tilting functor. The moduli $\mathcal M_p/\Q_p^\times$ of (local) geometric points coincides with the moduli space of all untilts of $C^\flat$ up to equivalence (\ie the set of the closed points of the Fargues--Fontaine curve and $C^\flat$ itself).
\end{enumerate}
\end{theorem}

\section{Points of $\spzf$ over $\C$}\label{sect4}

Having analyzed the geometric points of the $\Fone$-curve $\spzf=(\spz,\mathcal F)$ over perfectoid fields, we now investigate its evaluation over the field of complex numbers $\C$. At a fixed prime $p$, the absolute geometry naturally detects two places: the non-archimedean place $p$ and the archimedean place $\infty$. The space of non-trivial geometric points at $p$,
\[
\mathcal{M}_{p}:=\operatorname{Hom}_{\mathrm{loc}}(\Fone[T^{H_p^+}],H\C)\setminus\{1\},
\]
carries a natural symmetry structure whose precise description requires some care.

A first attempt to define a Frobenius-type action would be to let a complex number $w\in\C^\times$ act on a continuous local morphism
\[
\mathcal M_p\ni\rho:H_p^+\to\C^\times
\]
by the rule:
$(w\cdot\rho)(t)=\rho(t)^w$.
However, as soon as $w\notin\Z$, complex exponentiation becomes multivalued, rendering this expression ill-defined on $\C^\times$.

The key observation is that this ambiguity is canonically eliminated by the locality condition satisfied by $\rho$, namely continuity with respect to a place
$v\in\Sigma_{H_p}=\{p,\infty\}$.
As we shall see in this section, locality selects the appropriate branch structure and leads to a well-defined extension of the Frobenius symmetries.\vspace{0.03in}

In \S\ref{sectarch}, we prove that when $v=\infty$, the archimedean topology forces every local morphism $\rho$ to factor through the universal covering map
$
\exp:\C\longrightarrow\C^\times$.
Consequently, $\rho$ admits a unique lift of the form
$
\widetilde{\rho}(t)=zt$,
$z\in\C$, 
and the action of an element $w\in\C^\times$ becomes the ordinary scalar action on the Lie algebra:
$w\cdot\widetilde{\rho}(t)=(wz)t$.
Transporting this action through the exponential map yields a canonical and unambiguous action of the local Weil group
$
W_\infty=\C^\times
$
on the moduli space of archimedean points $\mathcal M_p^\infty$.

This action reveals a deeper geometric structure: the moduli spaces of points of $\spzf$ naturally arise as torsors (principal homogeneous spaces) under suitable Weil groups. More precisely, we shall see that the space
$\mathcal M_p^\infty$
of non-trivial archimedean points is parametrized by
\[
\widetilde{\C^\times}\setminus\{0\},
\]
which forms a canonical torsor over $W_\infty$. Two fundamental geometric quotients emerge from this structure:
\begin{enumerate}
\item \emph{The Tate curve $E_p$.}
The complex Tate curve
$E_p:=\mathcal M_p^\infty/p^\Z$
inherits the torsor structure and becomes a principal homogeneous space under the quotient group
$W_\infty/p^\Z$.

\item \emph{The projective line $\mathbb P^1(\R)$.}
The continuous symmetries of the completed stalk are encoded by the subgroup
$\R^\times=W_\infty^\sigma$,
where $\sigma\in\operatorname{Aut}(W_\infty)$ denotes complex conjugation. The quotient
\[
\mathbb P^1(\R)
\cong
(\widetilde{\C^\times}\setminus\{0\})/\R^\times
\]
therefore inherits a natural torsor structure under
$W_\infty/W_\infty^\sigma
\cong S^1$.
\end{enumerate}

In \S \ref{sectcpadic} we show that for the place $v=p$ the ambiguity is resolved canonically by the local Weil group $W_p=\Q_p^\times$, and that the non-trivial points  $\mathcal{M}_{p}^{p}$ form a single orbit of this group.

\subsection{Geometric $\C$-points  local for  $\infty\in \Sigma_{H_p}$}\label{sectarch}

 We first determine the moduli space $\mathcal M_p^\infty$ of  points local for the place $\infty$.
 By Proposition \ref{prop:monoid_structure_morphisms}, any such point corresponds uniquely to a monoid homomorphism:
\[
\rho : (H_p^+, +) \longrightarrow (\C, \cdot)
\]
which is continuous with respect to the archimedean topology on $H_p^+$, induced by the standard inclusion $H_p\subset \R$, and the standard Euclidean topology on $\C$.

\begin{proposition}\label{prop:archimedean_points_C}
$(i)$~The map 
\[
\rho:\C \to \Pt_p^\infty(\spzf,\C),\quad \rho(z) : (H_p^+, +) \longrightarrow (\C, \cdot), \ \ \rho(z)(t):=\exp(zt), \ \forall t \in H_p^+,
\] defines a canonical bijection of $\C$ with the  space of  $\C$-points of $\spzf$ at $p$, local for the archimedean place $\infty$.

$(ii)$~Under the bijection defined in $(i)$, the action of $\R^\times=\Aut(\R)$ corresponds to the scaling action on $\C$, and the Frobenius action of $p \in p^\Z$ on the stalk corresponds to multiplication by $p$.

$(iii)$~The moduli space $\mathcal{M}_{p}^{\infty}$ of non-trivial archimedean points  forms a canonical torsor over the local Weil group $W_\infty=\C^\times$.
\end{proposition}

\begin{proof}
$(i)$~Let $\rho : H_p^+ \to \C$ be a continuous monoid homomorphism. By definition, $\rho(0) = 1$. Because $\rho$ is continuous at $0$, there exists a neighborhood of $0$ in $H_p^+$ such that $\rho(t) \neq 0$ for all $t$ in this neighborhood. 

We claim that $\rho(t) \neq 0$ for all $t \in H_p^+$. Suppose for contradiction that $\rho(t_0) = 0$ for some $t_0 > 0$. For any integer $n \ge 1$, the homomorphism property implies:
\[
\left( \rho\left(\frac{t_0}{p^n}\right) \right)^{p^n} = \rho(t_0) = 0 \implies \rho\left(\frac{t_0}{p^n}\right) = 0.
\]
However, in the archimedean topology, the sequence $t_0/p^n$ converges to $0$ as $n \to \infty$. Continuity then demands that $\rho(t_0/p^n) \to \rho(0) = 1$, which contradicts $\rho(t_0/p^n) = 0$. Thus, $\rho$ never vanishes, and its image lies entirely in $\C^\times$.

Because $H_p^+$ is dense in $\R_{\ge 0}$, the continuous homomorphism $\rho$ extends uniquely to a continuous monoid homomorphism $\tilde{\rho} : \R_{\ge 0} \to \C^\times$. Furthermore, since $\R_{\ge 0}$ generates the group $\R$ and $\C^\times$ is a group, $\tilde{\rho}$ extends uniquely to a continuous group homomorphism from $\R$ to $\C^\times$. 

By standard Lie group theory, any continuous one-parameter subgroup of $\C^\times$ lifts through the universal covering map $\exp : \C \to \C^\times$. Therefore, there exists a unique complex number $z \in \C$ such that for all $t \in H_p^+$:
\[
\rho(t) = \exp(zt).
\]
This establishes a canonical bijection between the continuous morphisms $\rho$ and the parameter space $\C$.

Finally, we determine the action of the canonical symmetries. The  action of $\lambda \in \R^\times$ on the completed stalk is given by pre-composition with multiplication by $\lambda$. Acting on the morphism $\rho_z(t) = \exp(zt)$, we obtain:
\[
(\lambda \cdot \rho_z)(t) = \rho_z(\lambda \,t) = \exp(z(\lambda\,t)) = \exp((\lambda\,z)t) = \rho_{\lambda\,z}(t).
\]
This shows $(ii)$, in particular the Frobenius action on the stalk translates exactly to multiplication by $p$ on the parameter space $\C$.\newline
$(iii)$~By $(i)$, the scaling action of $W=\C^\times$ on $\mathcal{M}_{p}^{\infty}$ comes from the canonical action of $W$ on the universal cover of $\C^\times$. While the isomorphism given in $(i)$ depends on the choice of the covering $\exp:\C\to \C^\times$ the action of $W$ is independent of this choice, and $\mathcal{M}_{p}^{\infty}$ is a principal homogeneous space for this action.
\end{proof}

\begin{remark}\label{complexan}
It is important to emphasize that the moduli space of $\C$-points  is not merely a topological space, but it canonically inherits a complex analytic structure directly from the target field. For any $t \in H_p^+$, the pointwise evaluation defines a map from the moduli space of continuous homomorphisms to $\C$:
\[
\operatorname{ev}_t : \operatorname{Hom}_{\text{cont}}(H_p^+, \C) \longrightarrow \C, \quad \rho \longmapsto \rho(t).
\]
By demanding that these evaluation maps be holomorphic, the moduli space is canonically endowed with a complex analytic structure. Under the parameterization $\rho_z(t) = \exp(zt)$, the evaluation maps $\operatorname{ev}_t(z) = \exp(zt)$ are entire functions, demonstrating that this induced complex structure coincides exactly with the standard complex structure of the plane $\C$. 
\end{remark}

\subsubsection{Quotient by the automorphisms of the completed stalk}

At the archimedean place $\infty\in\Sigma_{H_p}$, the completion of the group $H_p$ is the real line $\R$, whose full group of continuous scaling symmetries is $\R^\times$. The completion of the stalk $H_p^+$ is the positive real line $\R_{\geq 0}$ and this reduces the symmetry group to the subgroup $\R_+^\times \subset \R^\times$.
It is natural to compare the geometric quotient
\[
\Pt_p^\infty(\spzf,\C)/\R_+^\times
\]
with the non-archimedean construction of the closed points of the Fargues–Fontaine curve. The trivial point of the moduli space corresponds to $z=0$, namely to the neutral element $\mathbf 1$. Removing this point leaves the moduli space of non-trivial points
\[
\mathcal M_p^\infty
=
\Pt_p^\infty(\spzf,\C)\setminus\{0\}
\cong
\C^\times.
\]

By Proposition~\ref{prop:archimedean_points_C}(ii), the symmetry group $\R^\times$ acts on the parameter space $\C^\times$ by ordinary real scalar multiplication. Taking the quotient by this full continuous symmetry group yields the orbit space
\[
\mathcal M_p^\infty/\R^\times
\cong
\C^\times/\R^\times.
\]

Geometrically, this quotient parametrizes the real lines through the origin in the plane $\C\cong\R^2$, and is therefore naturally identified with the real projective line
$\mathbb P^1(\R)$. But the automorphism group of the completed stalk is the subgroup $\R_+^\times\subset \R^\times$ and the analogue of the closed points of the Fargues-Fontaine curve is the double cover of $\mathbb P^1(\R)$ 
given by the following archimedean analogue of Corollary~\ref{luriecor4}.

\begin{corollary}\label{luriearch}
$(i)$~The quotient space $\mathcal{X}_\infty :=\mathcal{M}_{p}^{\infty} / \R^\times\cong \mathbb{P}^1(\R)$  forms a canonical torsor over $W / W^\sigma$.\newline
$(ii)$~The quotient space $\widetilde{\mathcal{X}}_\infty:=\mathcal{M}_{p}^{\infty} / \R_+^\times$ by the automorphism group of the completed stalk, is the  unramified double cover of the projective line   $\mathcal{X}_\infty  \cong \mathbb{P}^1(\R)$, with the deck transformation given by the  element $-1\in W_\infty$.
\end{corollary}

Thus, the same construction underlies both the non-archimedean and archimedean settings. Namely, after removing the trivial point and quotienting by the automorphism group of the completed stalk, one obtains the space of closed points of the Fargues–Fontaine curve in the perfectoid case, while over $\C$ one recovers the double cover $\tilde{\mathcal{X}}_\infty$ of the projective line. 

\subsubsection{The  Tate curve $E_p$}

The  Tate curve $E_p=\C^\times / p^\Z$ arises from quotienting the space of  non-trivial local morphisms $\mathcal{M}_{p}^{\infty}$  by the discrete arithmetic symmetry group $p^\Z$ of the uncompleted stalk $H_p$ generated by the  Frobenius action $z \mapsto pz$. This yields  the complex elliptic curve:
\[
E_p =\mathcal{M}_{p}^{\infty}/p^\Z\cong \C^\times / p^\Z.
\]
The isomorphism in display follows from Proposition \ref{prop:archimedean_points_C}, $(i)$. By $(iii)$ the space  $E_p = \mathcal{M}_{p}^{\infty}/ p^\Z$ forms a torsor over the quotient group $W_\infty / p^\Z$.

By lifting via the exponential map $w \mapsto \exp(2\pi i w)$, we see that $E_p \cong \C / \Lambda_p$, where the lattice is $\Lambda_p = \Z \oplus \Z \tau$, with modular parameter $\tau = i \frac{\log p}{2\pi}$. This is exactly the complex analytic realization of the Tate curve with parameter $q = 1/p$. Because $\tau$ is purely imaginary, $E_p$ is a rectangular torus defined over $\R$.\vspace{.03in}

It follows from Remark~\ref{complexan} that the space of non-trivial points $\mathcal{M}_{p}^{\infty}\cong\C^\times$ carries a canonical structure of complex manifold, and that the Frobenius action
$z\mapsto pz$
acts by bi-holomorphisms. Consequently, the quotient
$
E_p=\mathcal{M}_{p}^{\infty}/p^\Z
$
inherits a canonical structure of Riemann surface.\vspace{.03in}

This constitutes a significant refinement of the underlying $\Fone$-geometry. The absolute geometry produces not merely a topological torus, but a rigid complex-analytic object that intrinsically encodes the invariants of the associated elliptic curve, including the modular parameter
$\tau=i\frac{\log p}{2\pi}$
and the corresponding $j$-invariant. The canonical complex structure on $E_p$ therefore promotes the quotient naturally to an algebraic curve over $\C$.

Since $E_p$ is a compact Riemann surface, one may consider its field of global meromorphic functions. These are precisely the multiplicatively periodic meromorphic functions on $\C^\times$ satisfying
$f(pz)=f(z)$.
This field is an algebraic function field of transcendence degree one over $\C$, generated by the Weierstrass $\wp$-function and its derivative. The algebraic relations among these generators determine a projective algebraic curve, whose structure sheaf is recovered by considering meromorphic functions with prescribed finite sets of poles and passing to the corresponding localizations.

Thus, the evaluation of $\Fone$-morphisms yields considerably more than an analytic space: it canonically reconstructs the full algebraic curve structure of the Tate curve. Moreover, the torsor structure of $E_p$ under its group of translations is precisely the one induced by the action of
\[
W_\infty/p^\Z=\C^\times/p^\Z
\]
described above.

\subsection{The real structure}
The symmetries of the archimedean points are further enriched by the arithmetic of the target field $\C$. The complex conjugation $\sigma \in \operatorname{Gal}(\C/\R)$ acts naturally on the moduli space of geometric points by conjugating the target values: $\rho^\sigma(t) = \overline{\rho(t)}$. On the parameter space $\mathcal{M}_{p}^{\infty}\cong\C^\times$, this corresponds to standard complex conjugation $z \mapsto \bar{z}$. 

Crucially, because the Frobenius action (multiplication by $p$) is strictly real, it commutes with $\sigma$. Consequently, complex conjugation descends to an \emph{anti-holomorphic involution} on the Tate curve $E_p =\mathcal{M}_{p}^{\infty}/p^\Z\cong \C^\times / p^\Z$, endowing it with a canonical real structure. We may explicitly determine the real locus $E_p(\R)$ by finding the fixed points of $\sigma$ on the quotient. A class $[z]$ is fixed if and only if $\bar{z} = p^n z$ for some integer $n$. Taking absolute values forces $p^n = 1$, hence $n=0$ and so $z \in \R^\times$. 
The real locus of the Tate curve is therefore:
\[
E_p(\R) = \R^\times / p^\Z.
\]
Because $\R^\times$ decomposes multiplicatively into positive and negative reals, $\R^\times = \R_+^\times \sqcup \R_-^\times$, and the Frobenius action preserves the sign, the real locus splits into two disjoint connected components:
\[
E_p(\R) = (\R_+^\times / p^\Z) \sqcup (\R_-^\times / p^\Z).
\]
Geometrically, this reveals that the real points of the Tate curve consist of  two canonical copies of the periodic orbit $C_p = \R_+^\times / p^\Z$. 

\subsection{Rectangular decomposition of $E_p$}

The relationship between the Tate curve $E_p$, the periodic orbit $C_p$, and the analogue of the Fargues--Fontaine curve: 
$$
\widetilde{\mathcal{X}}_\infty:=\mathcal{M}_{p}^{\infty}/\R_+^\times,
$$ 
can be made completely explicit by analyzing the exact sequence of the underlying symmetry groups.
Consider the exact sequence associated to the positive real scaling symmetries:
\[
1 \to \R_+^\times \to \C^\times \to \C^\times/\R_+^\times \to 1.
\]
This sequence is canonically split by the absolute value map $z \mapsto |z|$, yielding the polar decomposition $\C^\times \cong \R_+^\times \times S^1$. 

Because the Frobenius generator $p$ is a positive real number, its action on $\C^\times$ is entirely confined to the $\R_+^\times$ factor. Consequently, the geometric quotient defining the Tate curve splits globally as a product of two real $1$-dimensional manifolds:
\[
E_p = \mathcal{M}_{p}^{\infty} / p^\Z = (\R_+^\times / p^\Z) \times \widetilde{\mathcal{X}}_\infty = C_p \times \widetilde{\mathcal{X}}_\infty.
\]
This canonical decomposition precisely isolates the arithmetic data from the geometric data. The first factor is  the periodic orbit $C_p$ (from the adelic scaling site), which carries the entire dependence on the prime $p$ (having length $\log p$). The second factor $\widetilde{\mathcal{X}}_\infty$ is a purely geometric phase circle that is totally independent of $p$, and has been described in Corollary \ref{luriearch}(ii).

The product decomposition $E_p=C_p \times \widetilde{\mathcal{X}}_\infty$ explains the origin of the rectangular complex structure of the Tate curve. Writing a point in $E_p$ in polar coordinates as $z = \lambda e^{i\theta}$, the canonical nowhere-vanishing holomorphic $1$-form on the elliptic curve is given by the logarithmic differential:
\[
\frac{dz}{z} = \frac{d\lambda}{\lambda} + i d\theta.
\]
This decomposition reveals that the complex structure on $E_p$ is uniquely and canonically determined by intertwining the invariant real $1$-forms (the Haar measures) of its two constituent Lie groups: namely, the invariant measure $d\lambda/\lambda$ on the arithmetic orbit $C_p$, and the invariant measure $d\theta$ on the geometric phase circle $\widetilde{\mathcal{X}}_\infty$ of length $2\pi$. The modular parameter of the elliptic curve, $\tau = i \frac{\log p}{2\pi}$, is therefore exactly the ratio of the integrals of these two canonical real differentials over their respective fundamental domains. We summarize this finding in the following:

\begin{thm}[Geometric Decomposition of the  Tate Curve]\label{thm:tate_decomposition}
The complex Tate curve $E_p$ of \eqref{tatec}, 
 admits a canonical global decomposition as a product of two real $1$-dimensional manifolds, perfectly isolating the arithmetic  from the geometric data:
\[
E_p \cong C_p \times \tilde{\mathcal{X}}_\infty.
\]
The first factor of this decomposition is the  periodic orbit $C_p = \R_+^\times / p^\Z$ of the adelic scaling site. It constitutes the connected component of the real locus of $E_p$ and carries the entire dependence on the prime $p$.  The second factor $\widetilde{\mathcal{X}}_\infty$ is the ($p$-independent) real analogue of the Fargues-Fontaine-curve.\newline
Furthermore, this product decomposition canonically determines the rectangular complex structure of the Tate curve. The canonical nowhere-vanishing holomorphic $1$-form on $E_p$ is given by:
\[
\omega = \frac{d\lambda}{\lambda} + i d\theta
\]
which uniquely intertwines the invariant real Haar measure $d\lambda/\lambda$ of the arithmetic orbit $C_p$ with the invariant real Haar measure $d\theta$ of the $p$-independent geometric factor $\widetilde{\mathcal{X}}_\infty$.
\end{thm}

\subsection{Geometric $\C$-points  local for  $p\in \Sigma_{H_p}$}\label{sectcpadic}

The absolute nature of the $\Fone$-stalk $\mathcal F_p=\Fone[T^{H_p^+}]$ ($H_p^+ = \Z[1/p]_+$) allows one to evaluate it over the complex numbers $\C$ using, this time, the $p$-adic topology on the stalk. This corresponds to looking at the non-archimedean place $p$ with values in $\C$. The change in topology on the domain (with respect to the archimedean case discussed before) completely transforms the moduli space, revealing a profound symmetry.

\begin{proposition}\label{prop:nonarchimedean_points_C}
Let  $H_p^+$ be endowed with the $p$-adic topology, and $\C$ with the standard Euclidean topology. The moduli space $\mathcal M_p^p$ of continuous local $\C$-points of $\spzf$  is canonically in bijection with the $p$-adic field $\Q_p$. Under this bijection, the Frobenius action of $p \in p^\Z$ on the stalk corresponds to multiplication by $p$ on $\Q_p$.
\end{proposition}

\begin{proof}
Let $\rho : H_p^+ \to \C$ be a continuous monoid homomorphism. The $p$-adic topology on $H_p^+$ dictates that $p^n \to 0$ as $n \to \infty$. Continuity therefore requires that $\rho(p^n) \to \rho(0) = 1\in \C$. 

Let $z = \rho(1) \in \C$. The continuity condition implies $z^{p^n} \to 1$. Writing $z = \exp(2\pi i \theta)$ for some $\theta \in \R/\Z$, this requires $p^n \theta \to 0 \pmod 1$. However, the multiplication-by-$p$ map on $\R/\Z$ strictly expands distances in a neighborhood of $0$. A sequence cannot converge to $0$ under a strictly expanding map unless it is eventually exactly $0$. Thus, there exists an integer $N$ such that $p^N \theta \equiv 0 \pmod 1$, meaning $\theta \in \Z[1/p]$. Consequently, $z$ is a $p$-power root of unity. 

Because $\rho(1/p^k)^{p^k} = z$, every element in the image of $\rho$ must also be a $p$-power root of unity. The image of $\rho$ thus lies entirely within $\mu_{p^\infty}(\C)$. Because the target of $\rho$ is a group, $\rho$ extends uniquely to a group homomorphism  $H_p = \Z[1/p]\to \mu_{p^\infty}(\C)$. 

Fixing a canonical isomorphism $\mu_{p^\infty}(\C) \cong \Q_p/\Z_p$, the moduli space is given by $$\operatorname{Hom}(H_p, \Q_p/\Z_p)\simeq \Q_p.$$ 

Finally, the Frobenius action of $p$ on the stalk is given by $(p \cdot \rho)(t) = \rho(pt)$. On the defining sequence, this yields $x_n \mapsto \rho(p/p^n) = x_{n-1} = p x_n$. Under the inverse limit isomorphism, this corresponds exactly to multiplication by $p$ on $\Q_p$.
\end{proof}

\section{Outlook}\label{outlook}

In \cite{CC3}, we developed the characteristic-$1$ geometry of the periodic orbit $C_p$ as an analogue of an elliptic curve. The identification of the absolute $\Fone$-geometry as a universal receptacle for the closed points of the Fargues–Fontaine curve naturally leads to a functorial connection with the characteristic-$1$ (equivalently, idempotent) geometry of the scaling site $\mathscr S$ introduced in \cite{CC3}.

Recall that the scaling site
$\mathscr S=([0,\infty)\rtimes\N^\times,\mathcal O)$
is a semiringed topos whose underlying space is the half-line $[0,\infty)$ endowed with the multiplicative action of $\N^\times$. Its structure sheaf $\mathcal O$ consists of continuous, convex, piecewise affine functions with integral slopes.

A remarkable feature of the analytic geometry of the Fargues–Fontaine curve is that it admits a canonical tropicalization into this characteristic-$1$ framework.

Indeed, analytic functions on the covering adic space of the Fargues–Fontaine curve are obtained by completing the ring (we follow the notation of \cite{FF})
$\mathbf A_{\inf}\!\left[\frac{1}{p},\frac{1}{[\varpi]}\right]$
with respect to a family of Gauss valuations $v_s$, where the parameter
$s\in(0,\infty)$
plays the role of a logarithmic radius. The behavior of these valuations is governed by a fundamental result of $p$-adic Hodge theory (\cf \cite[Lecture 11, Proposition 5]{Lurie}):

\begin{prop} Let $f$ be a nonzero element of $\mathbf{A}_{\inf}\left[\frac{1}{p}, \frac{1}{[\varpi]}\right]$. Then the map $s \mapsto v_s(f)$ determines a function
\[
v_{\bullet}(f): (0, \infty) \longrightarrow \R
\]
which is piecewise linear with integer slopes. Moreover, it is concave (that is, the slopes are decreasing).
\end{prop}

This proposition provides the analytic bridge to characteristic-$1$ geometry. Indeed, applying a global sign change yields the tropicalization map
$\tau(f)(s):=-v_s(f)$.
Since the valuation profile $v_\bullet(f)$ is concave, piecewise linear, and has integral slopes, it follows that $\tau(f)$ is convex, piecewise affine, and again has integral slopes.

Consequently, for every nonzero analytic function $f$ in the uncompleted period ring, the tropicalization $\tau(f)$ defines a global section of the characteristic-$1$ structure sheaf $\mathcal O$ of the scaling site $\mathscr S$. In this way, the family of logarithmic Gauss valuations gives rise to a canonical functor from characteristic-$0$ analytic geometry to the idempotent geometry of the scaling site.

Thus, the scaling site over the Boolean semifield $\B$ emerges as the canonical tropicalization of the covering space of the Fargues–Fontaine curve.\vspace{.05in}

A natural next step is to investigate the behavior of Frobenius eigenspaces, such as $B^{\varphi=p}$ in \cite{FF}, under descent to the periodic orbit
$C_p=\R_+^\times/p^\Z$,
and to clarify their relationship with the elliptic-curve geometry developed in \cite{CC3}. We leave this question for future work.


\begin{tabular}{p{0.48\textwidth} p{0.48\textwidth}}
\raggedright
Alain Connes\\
Coll\`ege de France\\
3 Rue d'Ulm\\
75005 Paris\\
France\\
\href{mailto:alain@connes.org}{alain@connes.org}
&
\raggedleft
Caterina Consani\\
Department of Mathematics\\
Johns Hopkins University\\
Baltimore, MD 21218\\
USA\\
\href{mailto:cconsan1@jhu.edu}{cconsan1@jhu.edu}
\end{tabular}
\end{document}